\documentclass[11pt, a4paper, oneside, reqno]{amsart}

\usepackage  [latin1 ]{inputenc}
\usepackage{amsmath, amsthm, amssymb, amsfonts}
\usepackage[foot]{amsaddr}
\usepackage{accents}
\usepackage[english]{babel}
\usepackage{bbding}
\usepackage{booktabs}
\usepackage[colorlinks=true, linkcolor = blue, citecolor = dgreen]{hyperref}
\usepackage{cleveref}
\usepackage{color}
\usepackage{crossreftools}
\usepackage{diagbox}
\usepackage{dsfont}
\usepackage[shortlabels]{enumitem}
\usepackage{eucal}
\usepackage{float}
\usepackage{framed}
\usepackage[a4paper]{geometry} 
\usepackage{graphicx}
\usepackage{mathrsfs}
\usepackage{multirow}
\usepackage[all]{xy} 
\usepackage{natbib}
\setcitestyle{round,comma,aysep={,},yysep={,}}
\usepackage{rotating}
\usepackage{setspace}
\usepackage{stmaryrd}
\usepackage{subcaption}
\usepackage{thmtools}
\usepackage[normalem]{ulem}
\usepackage{xspace}
\usepackage[dvipsnames]{xcolor}
\setlist{leftmargin=*, topsep=0.5em, parsep=0pt, itemsep=1em, labelindent=0pt, align=left}

\theoremstyle{definition}
\newtheorem{theorem}{Theorem}[section]

\newtheorem{definition}{Definition}[section]

\newtheorem{remark}{Remark}[section]

\newtheorem{assumption}{Assumption}%[section]
\crefname{assumption}{assumption}{assumptions}

\newtheorem*{assumption*}{Assumption}
\crefname{assumption*}{assumption}{assumptions}
\Crefname{assumption*}{Assumption}{Assumptions}

\numberwithin{equation}{section}

\newtheoremstyle{blue}%
{6.5pt}% Space above
{5pt}% Space below 
{}% Body font
{}% Indent amount
{\bfseries\color{blue}}% Theorem head font
{.}% Punctuation after theorem head
{.4em}% Space after theorem head 
{}% Theorem head spec (can be left empty, meaning ‘normal’)
\theoremstyle{blue}

\declaretheoremstyle[
spaceabove=6pt, spacebelow=6pt,
headfont=\normalfont\bfseries,
headpunct=,%
notefont=\normalfont\bfseries, notebraces={\hspace*{-4pt}}{},
bodyfont=\normalfont,
postheadspace=\newline, 
numbered=no,
]{mystyle}
\declaretheorem[style=mystyle,name={}]{myenv}

\newcommand*{\fullref}[1]{\Cref{#1} (\emph{\nameref*{#1}})}
\newcommand*{\fullrefs}[2]{Assumptions \labelcref{#1} (\emph{\nameref*{#1}}) and \labelcref{#2} (\emph{\nameref*{#2}})}
\newcommand*{\assref}[1]{\hyperref[{#1}]{(\emph{\nameref*{#1}})}}
\newcommand*{\partref}[1]{\labelcref{#1} (\emph{\nameref*{#1}})}

\newcommand{\littleo}[1]{o(#1)\xspace}

\newcommand{\Pro}{\mathbb{P}}

\newcommand{\Ex}{\mathbb{E}}

\newcommand{\EE}{\mathcal{E}}
\newcommand{\FF}{\mathcal{F}}

\newcommand{\HH}{\mathcal{H}}
\newcommand{\LL}{\mathcal{L}}

\newcommand{\PP}{\mathcal{P}}

\newcommand{\AAbb}{\mathbb{A}}

\newcommand{\ubar}{\underline}
\newcommand{\dd}{\mathrm{d}}

\newcommand{\alphabf}{\boldsymbol{\alpha}}
\newcommand{\bs}{\boldsymbol}

\newcommand{\reals}{\mathbb{R}}

\newcommand{\vep}{\varepsilon}

\newcommand{\tinT}{t\in[0,T]}

\newcommand{\rind}[1]{\textbf{1}_{#1}}

\numberwithin{table}{section}

\definecolor{red}{HTML}{D62728}
\definecolor{blue}{RGB}{ 0, 109, 219}
\definecolor{dgreen}{rgb}{0,.8,0}

\crefname{enumi}{}{}
\Crefname{enumi}{}{}

\crefformat{enumi}{#2#1#3}
\Crefformat{enumi}{#2#1#3}
\crefrangeformat{enumi}{#3#1#4--#5#2#6}
\Crefrangeformat{enumi}{#3#1#4--#5#2#6}

\crefmultiformat{enumi}{(#2#1#3)}{ and~(#2#1#3)}{, (#2#1#3)}{ and ~(#2#1#3)}
\Crefmultiformat{enumi}{(#2#1#3)}{ and~(#2#1#3)}{, (#2#1#3)}{ and ~(#2#1#3)}

\crefrangeformat{equation}{eqs. (#3#1#4)--(#5#2#6)}

\begin{document}
\title[]{Switching to a Green and sustainable finance setting: a mean field game approach}
\author{Anna Aksamit$^{\S}$}
\author{Kaustav Das$^{\dagger \ddagger}$}
\author{Ivan Guo$^{\dagger \ddagger}$}
\author{Kihun Nam$^{\dagger \ddagger}$}
\author{Zhou Zhou$^{\S}$}
\address{$^\dagger$School of Mathematics, Monash University, Victoria, 3800 Australia.}
\address{$^\ddagger$Centre for Quantitative Finance and Investment Strategies, Monash University, Victoria, 3800 Australia.}
\address{$^\S$School of Mathematics and Statistics, University of Sydney, New South Wales, 2006 Australia.}
\email{anna.aksamit@sydney.edu.au}
\email{kaustav.das@monash.edu}
\email{ivan.guo@monash.edu}
\email{kihun.nam@monash.edu}
\email{zhou.zhou@sydney.edu.au}

\date{}

\maketitle

\begin{abstract}

We consider a continuum of carbon-emitting firms who seek to maximise their stock price, and a regulator (e.g., Government) who wishes for the economy to flourish, whilst simultaneously punishing firms who behave non-green. Interpreting the regulator as a major player and the firms as the minor players, we model this setting through a mean field game with major and minor players. We extend the stochastic maximum principle derived by Carmona \& Zhu [A probabilistic approach to mean field games with major and minor players. Annals of Applied Probability, 2016, 94, 745--788] by relaxing the assumptions on the forms of the minimisers for the Hamiltonians, allowing them to depend on more arguments. This allows the major and representative minor player to interact in a more natural fashion, thereby permitting us to consider more realistic models for our green and sustainable finance problem. Through our stochastic maximum principle, we derive explicit Nash equilibria for a number of examples.

\vspace{0.5cm}

\hspace{-.42cm}Keywords: Mean Field Games; Game theory; FBSDE of McKean-Vlasov type, Major and Minor players; Green and sustainable finance

\end{abstract}

%---------------------------Introduction------------------------------------------
%---------------------------------------------------------------------------------
\section{Introduction}
\label{sec:introduction}
\noindent
%-----------Brief overview of paper
% \textcolor{Purple}{
% Consider a type of mean field game with a major and minor player. The game consists of a population of infinitely many indistinguishable and statistically identical players (the minor players), whose individual contributions to the system are negligible, and another separate player whose impact on the whole system is ever present and felt at all times (the major player). In the context of applications, the major player could be a regulator, whereas the population of minor players could be firms; we will use this example as motivation as well as inspiration for notation and terminology. Both the major player and representative minor player's objective is to minimise their total expected utility, in the sense of a sort of Nash equilibrium, over the set of their admissible strategies. 
% }
In this article, we consider a mean field description of a Green and sustainable finance problem, where we have a continuum of carbon-emitting firms who seek to maximise their stock price, and a regulator (e.g., Government) who wishes for the economy to thrive, whilst simultaneously punishing firms who behave non-green. We model this setting through a type of mean field game with a major and minor player, where the regulator is designated as the major player, and the firms as the minor players. To elaborate, such a game consists of a population of infinitely many indistinguishable and statistically identical players (the minor players), whose individual contributions to the system are negligible, and another separate player whose impact on the whole system is ever present and felt at all times (the major player). Both the major player and representative minor player's objective is to minimise their total expected utility, in the sense of a sort of Nash equilibrium, over the set of their admissible strategies. \citep{carmona2016probabilistic} examined this type of game in great detail, and under suitable assumptions proved both necessary and sufficient stochastic maximum principles for this setting, and moreover, proved convergence of approximate Nash equilibria in the finite player game to Nash equilibria in the mean field description. Despite this, a critical assumption they make is on the form of minimisers of the major and representative minor player's minimisers, namely, that they do not depend on certain arguments. We found however, that the class of models that obeys this assumption is rather small, and thus prevents one from considering many realistic models in applications. Therefore, our main focus in this article is to relax these assumptions considerably, and then in this new setting, derive similar results to theirs. This enables us to study the aforementioned regulator/firm Major/minor player game with non-trivial models.
%Here we have a number of firms who model their stock price, and a regulator who penalises those firms that behave non-green. 

%In order to include joint distribution of the representative minor player's state process and strategy, rather than utilising the results of extended Mean Field Games, we apply a new approach by expanding the state process in a suitable fashion.
%-----------MFGs
The theory of mean field games was first initiated by \citep{lions20061, lions20062, lasry2007mean} and independently by \citep{huang2003individual, huang2006large}. The main purpose of the theory is to study large systems of competitive individuals at a macroscopic level rather than the microscopic by enforcing the mean field ideology, namely that the number of players is large, they are statistically identical, their individual contributions are negligible, and each individual can only interact with the population through the overall presence of the population - the so-called mean field. The microscopic approach requires one to investigate a finite player stochastic differential game, a feat which is nigh impossible when the number of players is assumed to be large. Therefore, the mean field approach is highly enticing as it allows one to recast the system through a mean field description, which is typically easier to solve than the corresponding stochastic differential game formulation of the problem. By solving the mean field game, it is then possible to construct approximate Nash equilibria for the finite player stochastic differential game. This in turn justifies the theory of mean field games as a tractable way to tackle such systems and most importantly, form meaningful solutions to the finite player setting. Mathematically, there have been two primary directions for tackling mean field games, the original analytic approach developed by Lasry and Lions, which is reminiscent of the Dynamic Programming Principle and resulting HJB equation from control theory. The other approach which came later is called the probabilistic approach, spearheaded by Carmona and Delarue, see \citep{carmona2013probabilistic} and which relies on adaptations of the Pontryagin maximum principles, again from control theory. It would be remiss to not mention the books \citep{carmona2018mfgI, carmona2018mfgII}, which are a compendium on the subject of mean field games via both the analytic and probabilistic approach, however the focus is mainly on the probabilistic approach.

%-----------Types of Mean Field Games - in particular major/minor
There are classes of large systems consisting of competitive players where the mean field ideology is violated in at least one way. Our focus in this article is on the situation where we have a player whose individual contribution to the system is ever present, and does not become negligible no matter how many players are included in the system, and is also not statistically identical to the other players. It is appropriate to call this particular player the major player, and dub the others as minor players. This setting, dubbed mean field games with major and minor players, was first introduced in the linear-quadratic setting \citep{huang2006large} and have been extensively studied since, see for example \citep{huang2010large, nguyen2012linear, nguyen2012mean, nourian2013, carmona2017alternative}. However, there have been disagreements within the community as to what truly constitutes a major player in the mean field setting, or in other words, what is the right mean field strategy that describes a mean field game with major and minor players. The Major/minor player strategy detailed in \citep{carmona2016probabilistic} was proved to be an accurate description of this setting, as they demonstrated that the approximate Nash equilibria they construct for the finite player setting indeed converges to the Nash equilibria in the mean field setting. This is the interpretation that we adopt in this article. It is also worth noting that there have been efforts in describing a mean field setting for a Stackelberg game, for example \citep{bensoussan2015mean, bensoussan2016mean, elie2019tale, dayanikli2024machine}.

%-------------ESG papers
Mean field game theory has proved as an attractive method to describe settings regarding Environmental, Social, and Governance concerns, in particular scenarios pertaining to Green and sustainable finance. For example, \citep{lavigne2023decarbonization} consider a setting where there is a continuum of carbon-emitting firms, and two groups of investors, the green-minded investor and neutral investor. They establish existence of Nash equilibria for the mean field game amongst the firms and obtain results which have meaningful interpretations for the game they consider. In \citep{dayanikli2024multi}, they consider a setting with multiple populations of carbon-emitting electricity producers, where each population is at the behest of a taxation regulator (who acts as a principal as they are assumed to have no state process). Adapting the theory developed in \citep{carmona2016probabilistic}, they model this setting as a so-called multipopulation mean field game with major players and minor players. 
%However, due to the restrictive assumptions put forth on the forms of minimisers of the Hamiltonians, their...\textcolor{orange}{Finish this.}

%-----------Our contribution
Our main contribution in this article is an extension of the necessary and sufficient stochastic maximum principles from \citep{carmona2016probabilistic}, these are \Cref{thm:necessarySMPMm} and \Cref{thm:sufficientSMPMm} respectively. The primary attribute of the extension is that we permit the minimisers of the players' Hamiltonians to depend on more arguments than previously allowed in the literature, which we believe is crucial for the purposes of applications. This is evident in \Cref{sec:ESGapplications}, where we tackle a number of examples modelling the regulator/firm Major/minor player game previously described, which we believe can not be addressed with previous results in the literature. We hope that this extension will motivate and excite both practitioners and academics alike to develop more sophisticated models to address problems related to Green and sustainable finance.  

%-----------Section details
The sections are organised as follows: 
\begin{itemize}
    \item In \Cref{sec:preliminaries} we introduce the mean field game framework we will be working with in this article. We describe the Major/minor mean field game strategy that will be a central point in this article.
    \item In \Cref{sec:smp} we state and prove both a necessary and sufficient stochastic maximum principle for our Major/minor mean field game strategy.
    \item In \Cref{sec:enlargingstatespace} we adjust our model framework by enlarging the state process to include a component corresponding to the representative minor player's integrated strategy. This technique allows us to achieve results similar to that of extended mean field games.
    \item In \Cref{sec:ESGapplications} we consider examples of the regulator/firm Major/minor player game previously described, and derive explicit Nash equilibria.
\end{itemize}
\Cref{appen:Lderivative} contains some useful information regarding the differential calculus over spaces of probability measures that we choose to adopt in this article, the so-called L differential calculus, and is mainly a summary of ideas found in \citep[][Chapter 5]{carmona2018mfgI}. \Cref{appen:necessarySMPproof} contains the proof of \Cref{thm:necessarySMPMm}.

%-----------------------------Preliminaries---------------------------------------
%---------------------------------------------------------------------------------
\section{Prelimininaries}
\label{sec:preliminaries}
\noindent
All vectors in this article will be understood as column vectors. For a function $f: \reals^n \to \reals^m$, $\partial_x f$ will denote its Jacobian, which is an element of $\reals^{n \times m}$ for fixed $x \in \reals^n$. When $m = 1$, this will coincide with the gradient of $f$. For $a, b \in \reals^n$, we write $a \cdot b := \sum_{i = 1}^n a_i b_i$ to denote the usual Euclidean inner product on $\reals^n$. For matrices $A, B \in \reals^{n \times m}$, we write $A \cdot B = \text{tr}(A B^\top)$ to denote the Frobenius inner product. Sometimes when an argument in a function is redundant, we will put a $\circ$ in place of it.

Let $(E, \EE)$ be a measurable space where $(E, \| \cdot \|_E)$ is a normed space and $\EE$ the Borel $\sigma$-field generated by $\| \cdot \|_E$. In the following we will denote by $\PP_p(E)$, $p \geq 0$, the space of $p$-integrable probability measures, that is, probability measures on $\mu$ on $(E, \EE)$ that satisfy
\begin{align*}
	\int_E \|x\|_E^p \mu (\dd x) < \infty.
\end{align*}
The space $\PP_p(E)$ is endowed with the following $p$-Wasserstein metric: 
\begin{align*}
	W_p(\mu, \mu') := \left (\inf_{\pi \in \Pi(\mu, \mu')} \int_{E \times E} \|x - y\|_E^p \pi(\dd x \times \dd y) \right )^{1/p}
\end{align*}
where $\Pi(\mu, \mu')$ is the subset of $\PP_p(E \times E)$ with marginals $\mu$ and $\mu'$ in the first and second components respectively. The $p$-th moment of a measure $\mu \in \PP_p(E)$ is defined as 
\begin{align*}
	M_p(\mu) := \int_E \| x \|_E^p \mu(\dd x).
\end{align*}
Moreover, the mean of a probability measure $\mu$ is denoted by $\bar \mu := \int_E x \mu (\dd x)$, the integral being interpreted as a Bochner integral. In general, the mean of a probability measure does not coincide with its first moment.

For an $N$-tuple $x = (x^1, \dots, x^N) \in E^N$ we will write $x^{-i} = (x^1, \dots, x^{i-1}, x^{i+1}, \dots, x^N)$ to denote the tuple with the $i$-th element removed, and $(y, x^{-i}) = (x^1, \dots, x^{i-1}, y, x^{i+1}, \dots, x^N)$ so that $x = (x^i, x^{-i})$. Moreover, the empirical distributions of $x$ and $x^{-i}$ will be written as 
\begin{align*}
	\bar \mu_x &:= \frac{1}{N} \sum_{j =1}^N \delta_{x^j},
	&\bar \mu_{x^{-i}} &:= \frac{1}{N-1} \sum_{ 1 \leq j \neq i \leq N} \delta_{x^j},
\end{align*}
respectively, where $\delta_y$ denotes the Dirac measure on $(E, \EE)$.

%----------------Matrix notation
\begin{remark}[Matrix notation]
\label{remark:matrixnotation}
In the rest of this article, we will utilise the following notation regarding matrices:
The zero and identity matrices of dimension $p \times q$ will be denoted by $\bs{0}_{p \times q}$ and $I_{p \times q}$ respectively. For matrices $A \in \reals^{p \times c}, B \in \reals^{q \times c}$ we denote by $(A; B)$ the stacking of $A$ on top of $B$, and which takes values in $\reals^{p + q \times c}$. This only makes sense when the $A$ and $B$ have the same number of columns. Precisely, $(A; B)_{i, j} = A_{i,j}\rind{\{1 \leq i \leq p\}} + B_{i, j} \rind{\{p + 1 \leq i \leq p + q\}}$. 
For matrices $C \in \reals^{r \times p}, D \in \reals^{r \times q}$, we denote by $(C, D)$ the placing of $D$ onto the right of $C$, which takes values in $\reals^{r \times p + q}$. This only makes sense when the matrices have the same number of rows. Precisely, $(C, D)_{i,j} = C_{i,j} \rind{\{1 \leq j \leq p\}} + D_{i,j} \rind{\{p + 1 \leq j \leq p + q\}}$. It can seen that $(A; B) = (A^\top, B^\top)^\top$ and similarly, $(C, D) = (C^\top; D^\top)^\top$.

% Suppose $A \in \reals^{(m \times d) \times n}$ and $B \in \reals^{n \times m}$. Then $AB \in \reals^{(m \times d) \times m}$, and is defined as follows. First write $A = (A_{11}, \dots, A_{1d}, \dots, A_{md})$ and $B = (B_1, \dots, B_m)$, where each subelement (i.e., $A_{ij}, B_k$) is of dimension $n$. Then we define $(AB)_{ijk} =  A_{ij}\cdot B_k$, for $i = 1, \dots, m, j = 1, \dots, d, k = 1, \dots, m$. Similarly, suppose $A \in \reals^{m \times (n \times d)}$ and $B \in \reals^{n \times d}$. Then $AB \in \reals^{m \times 1}$, and is defined as follows. Write $A = (A_1, \dots, A_m)$ where each $A_i \in \reals^{n \times d}$. Then we define $(AB)_k = A_k \cdot B$, for $k = 1, \dots, m$.
\end{remark}
%---------------------------Model framework---------------------------------------
%---------------------------------------------------------------------------------
\subsection{Model framework}
\label{sec:modelframework}
Consider a finite time horizon $[0, T]$. Let $B = (B_t)_{\tinT}$ be a $m$-dimensional standard Brownian motion with augmented natural filtration $(\FF_t)_{\tinT}$. In the rest of this article, a subscript $0$ will refer to the major player, whereas the absence of a subscript will refer to the representative minor player. Let $A_0 \subseteq \reals^{k_0}$ and $A \subseteq \reals^k$ denote the set of actions for the major and a representative minor player respectively at a time $t \in [0, T]$, these sets being closed and convex. Denote by $\AAbb_0 $ and $\AAbb$ the set of strategy profiles for the major and minor player respectively. Here we assume that the major player's strategies are deterministic, whereas the representative minor player's strategy is progressively measurable w.r.t. $(\FF_t)_{\tinT}$. In either case, the strategy is assumed to be square integrable. For example, $\alpha = (\alpha_{t})_{\tinT} \in \AAbb$ if $\alpha : [0, T] \times \Omega \to A$, progressively measurable w.r.t. $(\FF_t)_{\tinT}$, and $\Ex \int_0^T |\alpha_t|^2 \dd t < \infty$, whereas $\alpha_0 \in \AAbb_0$ if $\alpha_0: [0, T] \to A_0$ and $\int_0^T |\alpha_{0t}|^2 \dd t < \infty$.

Consider the state processes of the major and representative minor player respectively:
%-------------State processes
\begin{align}
	\dd X_{0t}^{\mu} &= b_0(t, X_{0t}^{\mu}, \alpha_{0t}, \mu_t) \dd t, \label{eqn:stateprocessmajor} \\
	\dd X^{\mu}_t &= b(t, X_{0t}^{\mu}, X^{\mu}_t, \alpha_{0t}, \alpha_{t}, \mu_t) \dd t + \sigma(t, X^{\mu}_t) \dd B_t, \label{eqn:stateprocessminor}
\end{align}
as well as their McKean-Vlasov versions
\begin{align}
	\dd X_{0t} &= b_0(t, X_{0t}, \alpha_{0t}, \LL(X_t)) \dd t, \label{eqn:stateprocessmajorMKV} \\
	\dd X_t &= b(t, X_{0t}, X_t, \alpha_{0t}, \alpha_{t}, \LL(X_t)) \dd t + \sigma(t, X_t) \dd B_t, \label{eqn:stateprocessminorMKV}
\end{align}
where $b_0:[0, T] \times \reals^{d_0} \times A_0 \times \PP_2(\reals^d) \to \reals^{d_0}$, $b: [0, T] \times \reals^{d_0} \times \reals^d \times A_0 \times A \times \PP_2(\reals^d) \to \reals^d$ and $\sigma: [0, T] \times \reals^d \to \reals^{d \times m}$.

%\textcolor{orange}{This explanation is kind of repeated straight after the MmMFG strategy.}
There is ostensibly a peculiarity here, namely the difference between the state processes \crefrange{eqn:stateprocessmajor}{eqn:stateprocessminor} vs their McKean-Vlasov counterparts \crefrange{eqn:stateprocessmajorMKV}{eqn:stateprocessminorMKV}. In this setting, $X$ and $X^{\mu}$ both refer to the state process of the representative minor player, the main difference being that \cref{eqn:stateprocessminor} has an arbitrary flow of probability measures $(\mu_t)_{\tinT} \subseteq \PP_2(\reals^d)$ in the measure argument, whereas in \cref{eqn:stateprocessminorMKV} the law of the process being solved for appears in the measure argument. It turns out that in the mean field strategy, the choice of whether one considers $X$ versus $X^{\mu}$ is crucial. Roughly speaking, the major player performs its optimisation whilst considering the system \crefrange{eqn:stateprocessmajorMKV}{eqn:stateprocessminorMKV}, whereas the representative minor player performs its optimisation whilst considering the system \crefrange{eqn:stateprocessmajor}{eqn:stateprocessminor} for some fixed flow of probability measures $(\mu_t)_{\tinT}$.

%--------------MmMFG strategy
The following strategy, which we dub the Major/minor mean field game strategy (MmMFG strategy for short) outlines the overall mean field strategy for the major and representative minor players:
\begin{myenv}[Major/minor mean field game strategy (MmMFG strategy)]
\label{MmMFGstrategy}
\begin{enumerate}[label = \textbf{Step \arabic*}, ref =\textbf{\arabic*}]
\item[]
\item \label{MmMFG:Step1}
Fix a flow of probability measures $(\mu_t)_{\tinT} \subseteq \PP_2(\reals^d)$. 
\item \label{MmMFG:Step2}
Find a Nash equilibrium $(\hat \alpha_0^{\mu}, \hat \alpha^{\mu})$ to the following stochastic differential game:
\begin{align*}
	J_0(\alpha_0, \alpha) &= \Ex \left [\int_0^T f_0(t, X_{0t}, \alpha_{0t}, \LL(X_t)) \dd t + g_0(X_{0T}, \LL(X_T)) \right ], \\
	J^{\mu}(\alpha_0, \alpha) &= \Ex \left [ \int_0^T f (t, X_{0t}^{\mu}, X_t^{\mu}, \alpha_{0t}, \alpha_{t}, \mu_t) \dd t + g(X_{0T}^{\mu}, X_T^{\mu}, \mu_T) \right ],
\end{align*}
where we minimise $J_0$ over $\alpha_0 \in \AAbb_0$ and $J^{\mu}$ over $\alpha \in \AAbb$ subject to the following state process(es) of the major and representative minor player:
%where we minimise $J_0$ over $\alpha_0 \in \AAbb_0$ and $J^{\mu}$ over $\alpha \in \AAbb$ subject to the following state process(es) of the major player:
%\begin{align*}
%	\dd X_{0t}^{\mu} &= b_0(t, X_{0t}^{\mu}, \alpha_{0t}, \mu_t) \dd t, \\
%	\dd X_{0t} &= b_0(t, X_{0t}, \alpha_{0t}, \LL(X_t)) \dd t,
%\end{align*}
%and minor player
%\begin{align*}
%	\dd X_t &= b(t, X_{0t}, X_t, \alpha_{0t}, \alpha_{t}, \LL(X_t)) \dd t + \sigma(t, X_t) \dd B_t, \\
%	\dd X_t^{\mu} &= b(t, X_{0t}^{\mu}, X_t^{\mu}, \alpha_{0t}, \alpha_{t}, \mu_t) \dd t + \sigma(t, X_t^{\mu}) \dd B_t.
%\end{align*}
\begin{align*}
	\dd X_{0t}^{\mu} &= b_0(t, X_{0t}^{\mu}, \alpha_{0t}, \mu_t) \dd t, \\
	\dd X^{\mu}_t &= b(t, X_{0t}^{\mu}, X^{\mu}_t, \alpha_{0t}, \alpha_{t}, \mu_t) \dd t + \sigma(t, X^{\mu}_t) \dd B_t,
\end{align*}
as well as their McKean-Vlasov versions
\begin{align*}
	\dd X_{0t} &= b_0(t, X_{0t}, \alpha_{0t}, \LL(X_t)) \dd t, \\
	\dd X_t &= b(t, X_{0t}, X_t, \alpha_{0t}, \alpha_{t}, \LL(X_t)) \dd t + \sigma(t, X_t) \dd B_t.
\end{align*}
\item \label{MmMFG:Step3}
Solve the fixed point problem $\LL(\hat X^{\mu}_t) = \mu_t$. Here $\hat X^{\mu}_t$ refers to the state process $(X_t^{\mu})_{\tinT}$ along the Nash equilibrium $(\hat \alpha_0^{\mu}, \hat \alpha^{\mu})$.
\end{enumerate}
We call the fixed point $(\hat \mu_t)_{\tinT}$ a \emph{MmMFG equilibrium}.
\end{myenv}
We will say that the \hyperref[MmMFGstrategy]{MmMFG strategy} has a \emph{solution} if \textbf{Steps} \Crefrange{MmMFG:Step1}{MmMFG:Step3} can legitimately be carried out.

We note that the optimisation of the major player involves perturbations of $\alpha_0 \mapsto J_0(\alpha_0, \alpha)$. Moreover, $\alpha_0$ is the strategy of the major player, which by assumption should impact themself and the whole population of minor players. Thus perturbations of the nature $\alpha_0 \mapsto J_0(\alpha_0, \alpha)$ should induce changes in any measure arguments. In contrast to this, the optimisation of the representative minor player involves perturbations of $\alpha \mapsto J^{\mu}(\alpha_0, \alpha)$. However, since the population is assumed to be arbitrarily large, the representative minor player's individual contribution to the system is lost, and thus any perturbation of its strategy $\alpha$ should not affect the whole population; this is why we fix the measure flow as $\mu_t$ prior to optimisation. In the end however, we do want the measure $\mu_t$ to be the law of the state process, hence this leads to the fixed point step (\textbf{Step }\Cref{MmMFG:Step3}).

%-----Remark: No noise for Major player
\begin{remark}
In our modelling framework, the state processes \cref{eqn:stateprocessmajor,eqn:stateprocessmajorMKV} corresponding to the major player do not possess a noise term, unlike that in \citep{carmona2016probabilistic}. We have done this for two reasons: (i) for the simplicity, so that the subsequent analysis does not require conditional expectations w.r.t. the major player's noise, (ii) since the examples that we study in \Cref{sec:ESGapplications} do not in fact include a state process for the major player. Indeed, we could have completely left out the major player's state process from this article, however we chose to include it since as long as there is no noise term, the subsequent analysis poses no extra challenge. In addition, for simplicity we have assumed that the diffusion coefficient $\sigma$ is independent of the strategies of both the major and representative minor player $\alpha_0, \alpha$, as well as the probability measure $\mu$. Despite this, in the rest of the article, it is still possible to include a noise term for the major player's state process and allow $\sigma$ to depend on $\alpha_0, \alpha, \mu$, however modifications will need to be made in the lines of \citep{carmona2016probabilistic}.
\end{remark}

%--------------Hamiltonians
The following Hamiltonians will be pivotal for the analysis. For $(t, x_0, x, y_0, y, z, \alpha_0, \alpha, \mu) \in [0, T] \times \reals^{d_0} \times \reals^d \times \reals^{d_0} \times \reals^d \times \reals^{d \times m} \times A_0 \times A \times \PP(\reals^d)$:
\begin{align}
	H_0(t, x_0, x, y_0, y, z, \alpha_0, \alpha, \mu) &= b_0(t, x_0, \alpha_0, \mu) \cdot y_0 + b(t, x_0, x, \alpha_0, \alpha, \mu) \cdot y \nonumber \\
	&\quad + \sigma(t, x) z + f_0(t, x_0, \alpha_0, \mu), \label{eqn:Ham_0} \\
	H(t, x_0, x, y, z, \alpha_0, \alpha, \mu) &= b(t, x_0, x, \alpha_0, \alpha, \mu) \cdot y + \sigma(t, x) \cdot z \nonumber \\
	&\quad + f(t, x_0, x, \alpha_0, \alpha, \mu), \label{eqn:Ham}
\end{align}
as well as their reduced versions respectively:
\begin{align}
	H_0^R(t, x_0, x, y_0, y, \alpha_0, \alpha, \mu) &= b_0(t, x_0, \alpha_0, \mu) \cdot y_0 + b(t, x_0, x, \alpha_0, \alpha, \mu) \cdot y \nonumber \\
	&\quad+ f_0(t, x_0, \alpha_0, \mu), \label{eqn:Ham_0_reduced} \\
	H^R(t, x_0, x, y, \alpha_0, \alpha, \mu) &= b(t, x_0, x, \alpha_0, \alpha, \mu) \cdot y + f(t, x_0, x, \alpha_0, \alpha, \mu). \label{eqn:Ham_reduced}
\end{align}
We will also introduce the following functions: 
\begin{align}
	\HH_0(t, x_0, y_0, z, \alpha_0, \alpha, \nu) &:= \int_{\reals^{d + d}} H_0(t, x_0, x, y_0, y, z, \alpha_0, \alpha, \mu) \nu(\dd x \times \dd y), \label{eqn:ExHam_0}\\
	\HH_0^R(t, x_0, y_0, \alpha_0, \alpha, \nu) &:= \int_{\reals^{d + d}} H_0^R(t, x_0, x, y_0, y, \alpha_0, \alpha, \mu) \nu(\dd x \times \dd y), \label{eqn:ExHam_0_reduced}
\end{align}
where $\nu \in \PP(\reals^{d + d})$ and $\mu$ refers to the marginal in the first $d$ components, namely $\mu(\cdot) := \nu(\cdot \times \reals^d)$. This means that for a random vector $(U_t, V_t)$ taking values in $\reals^{d + d}$, that 
\begin{align*}
	\HH_0(t, x_0, y_0, z, \alpha_0, \alpha, \LL(U_t, V_t)) &= \Ex [ H_0(t, x_0, U_t, y_0, V_t, z, \alpha_0, \alpha, \LL(U_t))], \\
	\HH_0^R(t, x_0, y_0, \alpha_0, \alpha, \LL(U_t, V_t)) &= \Ex [ H_0^R(t, x_0, U_t, y_0, V_t, \alpha_0, \alpha, \LL(U_t))].
\end{align*}
For simplicity, we may refer to $\HH_0$ (respectively $\HH_0^R$) as a Hamiltonian (respectively reduced Hamiltonian) as an abuse of terminology.
We will extensively utilise the notation $\ubar{x} := (x_0, x)$.

We will enforce the following assumption regarding the regularity of the drift, diffusion, running cost, and terminal cost coefficients above. For notational simplicity, in the following assumption, we will sometimes suppress arguments. As norms on finite dimensional spaces are equivalent, in the following assumption there is no need to worry about the precise definitions of any of the product norms. We refer the reader to \Cref{appen:Lderivative} for definitions of L differentiability of functions of probability measures and associated properties.
%------Assumption: Regularity of coefficients
\begin{assumption}[Regularity of coefficients]
\label{ass:regularity}
\begin{enumerate}[label = (A\arabic*), ref = A\arabic*]
\item[]
\item 
There exists a constant $c > 0$ such that for all $t \in [0, T]$, $x_0', x_0 \in \reals^{d_0}$, $\alpha_0', \alpha_0 \in A_0$, $x', x \in \reals^d$, $\alpha', \alpha \in A$, $\mu', \mu \in \PP_2(\reals^d)$, we have:
\begin{align*}
	&|(b_0, \sigma_0)(t, x_0', \alpha_0', \mu') - (b_0, \sigma_0)(t, x_0, \alpha_0, \mu) | + |(b, \sigma)(t, x_0', x', \alpha_0', \alpha', \mu') - (b, \sigma)(t, x_0, x, \alpha_0, \alpha, \mu) | \\
	&\leq  L \big ( |x_0'  - x_0| + |x'  - x| + |\alpha_0'  - \alpha_0| + |\alpha'  - \alpha| + W_2(\mu', \mu) \big ).
\end{align*}
\item For any $\alpha_0 \in A_0$ and $\alpha \in A$:
\begin{align*}
	\int_0^T \left (|(b_0, \sigma_0)(t, 0, \alpha_0, \delta_0)|^2 + |(b, \sigma)(t, 0, 0, \alpha_0, \alpha, \delta_0)|^2 \right ) \dd t < \infty.
\end{align*}
\item There exists a constant $c_L > 0$ such that for all $t \in [0, T]$, $x_0', x_0 \in \reals^{d_0}$, $\alpha_0', \alpha_0 \in A_0$, $x', x \in \reals^d$, $\alpha', \alpha \in A$, $\mu', \mu \in \PP_2(\reals^d)$, we have:
%\begin{align*}
%	&|(f_0, g_0)(t, \grave{x}_0, \grave{\alpha}_0, \grave{\mu}) - (f_0, g_0)(t, x_0, \alpha_0, \mu)| + |(f, g)(t, \grave{x}_0, \grave{x}, \grave{\alpha}_0, \grave{\alpha}, \grave{\mu}) - (f, g)(t, x_0, x, \alpha_0, \alpha, \mu) | \\
%	&\leq L \big ( 1 + |\grave{x}_0| + |x_0| + |\grave{x}| + |x| + |\grave{\alpha_0}| + |\alpha_0| +  |\grave{\alpha}| + |\alpha| + M_2(\grave{\mu}) + M_2(\mu) \big ) \\
%	&\quad \times \big ( |\grave{x}_0  - x_0| + |\grave{x} - x| + | \grave{\alpha}_0 - \alpha_0| + |\grave{\alpha} - \alpha| + W_2(\grave{\mu}, \mu) \big).
%\end{align*}
\begin{align*}
	&|(f_0, g_0)(t, x_0', \alpha_0', \mu') - (f_0, g_0)(t, x_0, \alpha_0, \mu)| \\
	&\leq c_L \big ( 1 + |x_0'| + |x_0| + |\alpha_0'| + |\alpha_0| + M_2(\mu') + M_2(\mu) \big ) \\
	&\quad \times \big ( |x_0'  - x_0|  + |\alpha_0' - \alpha_0| + W_2(\mu', \mu) \big).
\end{align*}
and 
\begin{align*}
	&|(f, g)(t, x_0, x', \alpha_0, \alpha', \mu') - (f, g)(t, x_0, x, \alpha_0, \alpha, \mu) | \\
	&\leq c_L \big ( 1 + |x_0'| + |x_0| + |x'| + |x| + |\alpha_0'| + |\alpha_0| +  |\alpha'| + |\alpha| + M_2(\mu') + M_2(\mu) \big ) \\
	&\quad \times \big ( |x_0'  - x_0| + |x' - x| + |\alpha_0' - \alpha_0| + |\alpha' - \alpha| + W_2(\mu', \mu) \big).
\end{align*}

\item The maps $(x_0, \alpha_0) \mapsto b_0, b, f_0$ are differentiable, and the maps $(x_0, \alpha_0, \mu) \mapsto \partial_{x_0} (b_0, b, f_0)$ and $(x_0, \alpha_0, \mu) \mapsto \partial_{\alpha_0}(b_0, b, f_0)$ are continuous. The maps $(x, \alpha) \mapsto b, \sigma, f$ are differentiable, and the maps $(x, \alpha, \mu) \mapsto \partial_x(b, \sigma, f)$ and $(x, \alpha, \mu) \mapsto \partial_\alpha(b, f)$ are continuous. 

The maps $\mu \mapsto b_0, b, f_0$ are L differentiable, and the maps $(x_0, X, \alpha_0) \mapsto \partial_\mu (b_0, f_0)(t, x_0, \alpha_0, \LL(X))(X)$ are continuous.

The map $x_0 \mapsto g_0$ is differentiable, and the map $(x_0, \mu) \mapsto \partial_{x_0} g_0$ is continuous. The map $\mu \mapsto g_0$ is L differentiable, and the map $(x_0, X) \mapsto \partial_\mu g_0(x_0, \LL(X))(X)$ is continuous. The map $x \mapsto g$ is differentiable, and the mapping $(x_0, x, \mu) \mapsto \partial_x g$ is continuous. 

% \item \textcolor{orange}{Boundedness condition... similar to \citep[][(A5)]{carmona2016probabilistic}.}
\end{enumerate}

\end{assumption}

%-----------------------Finite player setting------------------------------------
%---------------------------------------------------------------------------------
\subsection{Finite player setting}
\label{sec:finiteplayersetting}
The \hyperref[MmMFGstrategy]{MmMFG strategy} is meant to serve as a mean field description of the following finite player framework, which is described through a stochastic differential game. Suppose we have $N+1$ players, where the $0$-th player corresponds to the major player, and players $1, \dots, N$ are the minor players. Let $\bs{B} = (B^1, \dots, B^N)$ denote a $Nd$-dimensional Brownian motion. Denote by $\AAbb_i$ the space of admissible strategies for the $i$-th minor player, namely strategies $\alpha^i: [0, T] \times \Omega \to A$ which are progressively measurable w.r.t. the augmented filtration generated by $B^i$, and $\Ex[\int_0^T |\alpha_t^i|^2 \dd t] < \infty$. The major player's admissible strategies remain elements of $\AAbb_0$. Suppose the players possess the following state processes:
\begin{align*}
	\dd X_{0t} &= b_0(t, X_{0t}, \alpha_{0t}, \bar \mu_{\bs{X}_t}) \dd t, \\
	\dd X_t^i &= b(t, X_{0t}, X_t^i, \alpha_{0t}, \alpha_t^i, \bar \mu_{\bs{X}_t^{-i}}) \dd t + \sigma(t, X_t^i) \dd B_t^i, \quad X_0^i = x_0^i, \quad i = 1, \dots, N,
\end{align*}
where $X_0$ refers to the state process of the major player, and $X^i$ refers to the state process of the $i$-th minor player. Moreover, we write $\bs{X} := (X^1, \dots, X^N)$ and $\alphabf := (\alpha^1, \dots, \alpha^N)$.

The utility function for the major player is given by:
\begin{align*}
	J_0(\alpha_0, \alphabf) &=  \Ex \left [ \int_0^T f_0(t, X_{0t}, \alpha_{0t}, \bar \mu_{\bs{X}_t}) \dd t + g_0(X_{0T}, \bar \mu_{\bs{X}_T} )\right ],
\end{align*}
and for the $i$-th minor player:
\begin{align*}
	J^i(\alpha_0, \alpha^i, \alphabf^{-i}) &= \Ex \left [ \int_0^T f(t, X_{0t}, X_t^i, \alpha_{0t}, \alpha_t^i, \bar \mu_{\bs{X}_t^{-i}}) \dd t + g(X_{0T}, X_T^i, \bar \mu_{\bs{X}_T^{-i}}) \right ], \quad i = 1, \dots, N.
\end{align*} 
The objective of the above $N+1$ player game is for the major player to minimise $\alpha_0 \mapsto J_0(\alpha, \bs{\alpha})$ and the minor players to minimise $\alpha^i \mapsto J^i(\alpha_0, \alpha^i, \bs{\alpha}^{-i})$ in the sense of Nash equilibrium. However, finding true Nash equilibria to this stochastic differential game is difficult. Put simply, this is because the finite player setting for large $N$ is highly interactive, and thus the corresponding mathematical system (be it PDE or FBSDE in nature) becomes highly coupled. Thus, the purpose of passing to the mean field regime is to increase tractability of the system, and by solving the \hyperref[MmMFGstrategy]{MmMFG strategy}, determine approximate Nash equilibria for the finite player setting. What we precisely mean by approximate Nash equilibria is given by the following definition, which is originally stated in \citep[][\S 4]{carmona2016probabilistic}:
%----------Defn: epsilon Nash equilibrium
\begin{definition}[$\vep$-Nash equilibrium]
\label{defn:epsilonNash}
Fix $\vep > 0$. A set of admissible controls $(\hat \alpha_0, \hat \alpha^1, \dots, \hat \alpha^N) \in \AAbb_0 \times \prod_{i = 1}^N \AAbb_i$ is called an $\vep$-Nash equilibrium for the above $(N+1)$-player stochastic differential game if for all $\alpha_0 \in \AAbb_0$ we have 
\begin{align*}
    J_0(\hat \alpha_0, \hat{\bs{\alpha}}) - \vep \leq J_0(\alpha_0, \hat{\bs{\alpha}})
\end{align*}
and for all $i = 1, \dots, N$ and $\alpha^i \in \AAbb_i$ we have 
\begin{align*}
    J^i(\hat \alpha_0, \hat \alpha^i, \hat{\bs{\alpha}}^{-i}) - \vep \leq J^i(\hat \alpha_0, \alpha^i, \hat{\bs{\alpha}}^{-i}).
\end{align*}
\end{definition}
In general, the purpose of studying the mean field description of a system is to construct $\vep_N$-Nash equilibria to the associated stochastic differential game, where $\vep_N$ is a sequence which tends to $0$ as $N \to \infty$ in a sufficiently fast manner. How one can precisely determine $\vep_N$-Nash equilibria in our setting is described in \Cref{sec:approximatenashequilibria}.

%-------------------------------------SMP-----------------------------------------
%---------------------------------------------------------------------------------
\section{A stochastic maximum principle}
\label{sec:smp}
\noindent
We now derive a stochastic maximum principle (SMP) for the \hyperref[MmMFGstrategy]{MmMFG strategy}. As is standard, we first consider necessary conditions. From there, we will be able to deduce extra assumptions that lead to sufficiency in \Cref{sec:sufficientSMP}.
%---------------------------------Necessary SMP-----------------------------------
%---------------------------------------------------------------------------------
\subsection{A necessary stochastic maximum principle}
\label{sec:necessarySMP}
%------Assumption: Necessary SMP convexity
\begin{assumption}[Convexity of Hamiltonians in strategy]
\label{ass:necessarysmp}
$\alpha_0 \mapsto \HH_0(t, x_0, y_0, z,  \alpha_0, \alpha, \nu)$ and 
$\alpha \mapsto H(t, x_0, x, y, z, \alpha_0, \alpha, \mu)$ are convex.
% \begin{enumerate}[label = (B\arabic*), ref = B\arabic*]
% \item[]
% \item
% The map $H_0$ is differentiable in $(x_0, x, \alpha_0, \mu)$ and $\alpha_0 \mapsto \HH_0(t, x_0, y_0, z,  \alpha_0, \alpha, \nu)$ is convex. %$\alpha_0 \mapsto H_0(t, x_0, x, y_0, y, z, \alpha_0, \alpha, \mu)$ is convex.
% \item 
% The map $H$ is differentiable in $(x, \alpha)$ and $\alpha \mapsto H(t, x_0, x, y, z, \alpha_0, \alpha, \mu)$ is convex.
% \end{enumerate}
\end{assumption}

As we are deducing a necessary condition, we now assume that $(\hat \alpha_0^{\mu}, \hat \alpha^{\mu})$ is a Nash equilibrium, that is \textbf{Steps} \Crefrange{MmMFG:Step1}{MmMFG:Step2} of the \hyperref[MmMFGstrategy]{MmMFG strategy} are satisfied. At the moment, \textbf{Step} \Cref{MmMFG:Step3} is not yet considered.

%-----------------Necessary SMP FBSDEs
Consider the following FBSDEs. For the major player's problem:
\begin{align}
\label{eqn:necessarySMPFBSDEmajor}
\begin{split}
	\dd X_{0t} &= b_0(t, X_{0t}, \alpha_{0t}, \LL(X_t)) \dd t, \\
	\dd X_t &= b(t, \ubar{X}_t,  \alpha_{0t}, \alpha_{t}, \LL(X_t)) \dd t + \sigma(t, X_t) \dd B_t, \\
	\dd P_{0t} &= - \partial_{x_0} H_0(t, \ubar{X}_t, \ubar{P}_t,  Q_t, \alpha_{0t}, \alpha_t, \LL(X_t)) \dd t \\
                    &\quad + Q_{0t} \dd B_t, \\
	\dd P_t &= - \Big [\partial_x H_0(t, \ubar{X}_t, \circ, P_t, Q_t, \alpha_{0t}, \alpha_t, \LL(X_t)) \\
			&\quad + \tilde \Ex[ \partial_{\mu} H_0(t, \ubar{\tilde X}_t, \ubar{\tilde P}_t, \tilde Q_t, \tilde \alpha_{0t}, \tilde \alpha_{t}, \LL(X_t))(X_t)] \Big] \dd t\\
			&\quad + Q_t \dd B_t, \\
	P_{0T} &=  \partial_{x_0} g_0(X_{0T}, \LL( X_T)) (X_T), \\
	P_{T} &= \partial_{\mu} g_0(X_{0T}, \LL( X_T)) (X_T).
\end{split}
\end{align}
For the representative minor player's problem:
\begin{align}
\begin{split}
\label{eqn:necessarySMPFBSDEminor}
	\dd X_{0t}^{\mu} &= b_0(t, X_{0t}^{\mu}, \alpha_{0t}, \mu_t) \dd t, \\
	\dd X^{\mu}_t &= b(t, \ubar{X}^{\mu}_t, \alpha_{0t}, \alpha_{t}, \mu_t) \dd t + \sigma(t, X^{\mu}_t) \dd B_t, \\
	\dd Y^{\mu}_t &= - \partial_x H(t, \ubar{X}^{\mu}_t, Y_{t}^{\mu}, Z_{t}^{\mu},  \alpha_{0t}, \alpha_t, \mu_{t}) \dd t \\
			&\quad + Z^{\mu}_t \dd B_t, \\
		Y^{\mu}_T &=  \partial_x g(X_{0T}^{\mu}, X^{\mu}_T, \mu_{T}).
\end{split}
\end{align}

%------------Theorem: Necessary SMP
\begin{theorem}[Necessary stochastic maximum principle for MmMFG strategy]
\label{thm:necessarySMPMm}
Enforce \fullrefs{ass:regularity}{ass:necessarysmp}. Let \textbf{Steps} \Crefrange{MmMFG:Step1}{MmMFG:Step2} of the \hyperref[MmMFGstrategy]{MmMFG strategy} be fulfilled, and denote by $(\hat \alpha^\mu_0, \hat \alpha^\mu)$ the associated Nash equilibrium. Let $(\hat {\ubar{X}}, \hat {\ubar{P}}, \hat{\ubar{Q}})$ and $(\hat {\ubar{X}}^\mu, \hat Y^\mu, \hat Z^\mu)$ denote the unique solutions to the FBSDEs \cref{eqn:necessarySMPFBSDEmajor} and \cref{eqn:necessarySMPFBSDEminor} respectively along the Nash equilibrium $(\hat \alpha_0^\mu, \hat \alpha^\mu)$. Then
\begin{align*}
	\Ex [ H_0^R(t, \hat{\ubar{X}}_t, \hat {\ubar{P}}_{t}, \hat \alpha_{0t}^\mu, \hat \alpha_{t}^\mu, \LL(\hat X_t)) ] &\leq \Ex [ H_0^R(t, \hat{\ubar{X}}_t, \hat {\ubar{P}}_{t}, \alpha_{0}, \hat \alpha_{t}^\mu, \LL(\hat X_t)) ]
\end{align*}
for all $\alpha_0 \in A_0$, $\dd t$ a.e. and
\begin{align*}
	H^R(t, \hat {\ubar{X}}_t^\mu, \hat Y^\mu_{t}, \hat \alpha_{0t}^\mu, \hat \alpha_{t}^\mu, \mu_t) &\leq H^R(t, \hat {\ubar{X}}_t^\mu, \hat Y_t^\mu, \hat \alpha^\mu_{0t}, \alpha, \mu_{t})
\end{align*}
for all $\alpha \in A$, $\dd \Pro \times \dd t$ a.e.
\end{theorem}
%
%\begin{theorem}[Necessary stochastic maximum principle for MmMFG strategy]
%\label{thm:necessarySMPMm}
%\begin{enumerate}[label = (\roman*)]
%\item[]
%\item Let \textbf{Steps} \Crefrange{MmMFG:Step1}{MmMFG:Step2} of the MmMFG strategy be fulfilled, and denote by $(\hat \alpha^\mu_0, \hat \alpha^\mu)$ the associated Nash equilibrium.
%\item Let $(\hat {\ubar{X}}, \hat {\ubar{P}}, \hat{\ubar{Q}})$ and $(\hat {\ubar{X}}^\mu, \hat Y^\mu, \hat Z^\mu)$ denote the unique solutions to the FBSDEs \cref{eqn:necessarySMPFBSDEmajor} and \cref{eqn:necessarySMPFBSDEminor} respectively along the Nash equilibrium $(\hat \alpha_0^\mu, \hat \alpha^\mu)$.
%\end{enumerate}
%Then
%\begin{align*}
%	\Ex [ H_0^R(t, \hat{\ubar{X}}_t, \hat {\ubar{P}}_{t}, \hat \alpha_{0t}^\mu, \hat \alpha_{t}^\mu, \LL(\hat X_t)) ] &\leq \Ex [ H_0^R(t, \hat{\ubar{X}}_t, \hat {\ubar{P}}_{t}, \alpha_{0}, \hat \alpha_{t}^\mu, \LL(\hat X_t)) ]
%\end{align*}
%for all $\alpha_0 \in A_0$, $\dd t$ a.e. and
%\begin{align*}
%	H^R(t, \hat {\ubar{X}}_t^\mu, \hat Y^\mu_{t}, \hat \alpha_{0t}^\mu, \hat \alpha_{t}^\mu, \mu_t) &\leq H^R(t, \hat {\ubar{X}}_t^\mu, \hat Y_t^\mu, \hat \alpha^\mu_{0t}, \alpha, \mu_{t})
%\end{align*}
%for all $\alpha \in A$, $\dd \Pro \times \dd t$ a.e.
%\end{theorem}
\begin{proof}
Given in \Cref{appen:necessarySMPproof}.
\end{proof}
%-------------------------------Sufficient SMP------------------------------------
%---------------------------------------------------------------------------------
\subsection{A sufficient stochastic maximum principle}
\label{sec:sufficientSMP}
\noindent
A sufficient part of the stochastic maximum principle can be established under additional convexity conditions.
%----------Assumption: Convexity of running costs
\begin{assumption}[Convexity of Hamiltonians and terminal costs]
\label{ass:convexity}
\;
\begin{enumerate}[label = (C\arabic*), ref = (C\arabic*)]
\item
\label{ass:convexity1}
The Hamiltonians $H_0$ and $H$ satisfy the following convexity assumptions. For all $x_0, x_0' \in \reals^{d_0}$, $x, x' \in \reals^d$, $\alpha_0, \alpha_0' \in A_0$, $\alpha, \alpha' \in A$, $\mu, \mu' \in \PP_2(\reals^d)$, as well as $t \in [0, T], y_0 \in \reals^{d_0}, y \in \reals^d$, there exists $\lambda_0, \lambda > 0$ such that
\begin{align*}
	&H_0(t, x_0', x', y_0, y, z, \alpha_0', \alpha, \mu') - H_0(t, x_0, x, y_0, y, \alpha_0, \alpha, \mu) \geq \partial_x H_0 \cdot (x' - x) \\&\quad + \partial_{x_0} H_0 \cdot (x_0' - x_0)  + \partial_{\alpha_0} H_0\cdot (\alpha_0' - \alpha_0) + \tilde \Ex[ \partial_\mu H_0(\tilde X)\cdot (\tilde X' - \tilde X)] + \frac{1}{2} \lambda_0| \alpha_0' - \alpha_0|^2, \\
	&H(t, x_0', x', y, z, \alpha_0, \alpha', \mu') - H(t, x_0, x, y, \alpha_0, \alpha, \mu) \geq \partial_{x_0} H \cdot (x_0' - x_0) \\&\quad + \partial_x H \cdot (x' - x) + \partial_\alpha H \cdot (\alpha' - \alpha) + \tilde \Ex[ \partial_\mu H(\tilde X) \cdot (\tilde X' - \tilde X)] + \frac{1}{2} \lambda| \alpha' - \alpha|^2,
\end{align*}
where $\tilde X, \tilde X'$ are defined on some probability space $(\tilde \Omega, \tilde \FF, \tilde \Pro)$, square integrable, and possess laws $\mu, \mu'$ respectively. Here we write $H_0 \equiv H_0(t, x_0, x, y_0, y, z, \alpha_0, \alpha, \mu)$ and $H \equiv H(t, x_0, x, y, z, \alpha_0, \alpha, \mu)$ for simplicity.
\item
The terminal costs $g_0$ and $g$ satisfy the following convexity conditions. For all $x_0, x_0' \in \reals^{d_0}$, $x, x' \in \reals^d$, $\mu, \mu' \in \PP_2(\reals^d)$,
% The terminal costs $g_0$ and $g$ satisfy the following convexity conditions in $(x_0, \mu)$ and $x$ respectively: 
\begin{align*}
	g_0(x_0', \mu') - g_0(x_0, \mu) &\geq \partial_{x_0} g_0(x_0, \mu)\cdot (x_0' - x_0) + \tilde \Ex [ \partial_\mu g_0(x_0, \mu)(\tilde X) \cdot (\tilde X' - \tilde X)], \\
	g(x_0, x', \mu) - g(x_0, x, \mu) &\geq \partial_x g(x_0, x, \mu) \cdot (x' - x),
\end{align*}
where $\tilde X, \tilde X'$ are defined on some probability space $(\tilde \Omega, \tilde \FF, \tilde \Pro)$, square integrable, and possess laws $\mu, \mu'$ respectively.
\end{enumerate}
\end{assumption} 
Clearly \Cref{ass:convexity} encompasses \Cref{ass:necessarysmp}.
%----------Assumption: Minimisation of Hamiltonians
\begin{assumption}[Minimisers of Hamiltonians]
\label{ass:minimisersofhamiltonians}
Functions 
\begin{align*}
    [0, T] \times \reals^{d_0} \times \reals^{d_0} \times \PP_2(\reals^{d + d}) \ni (t, x_0, y_0, \nu) \mapsto \mathring \alpha_0(t, x_0, y_0, \nu) \equiv \mathring \alpha_{0t}, \\
    [0, T] \times \reals^{d_0} \times \reals^d \times \reals^d \times \PP_2(\reals^{d}) \ni (t, x_0, x, y, \mu) \mapsto \mathring \alpha(t, x_0, x, y, \alpha_0, \mu) \equiv \mathring \alpha_t
\end{align*} 
uniquely exist such that
\begin{align}
	\HH_0^R(t, x_0, y_0, \mathring \alpha_{0t}, \alpha, \nu) &\leq \HH_0^R(t, x_0, y_0, \alpha_0, \alpha, \nu), \quad \forall \alpha_0 \in A_0, \label{eqn:HH0minimisation}\\
	H^R(t, x_0, x, y, \alpha_0, \mathring \alpha_t, \mu) &\leq H^R(t, x_0, x, y, \alpha_0, \alpha, \mu), \quad \forall \alpha \in A. \label{eqn:Hminimisation}
\end{align}
\end{assumption}
Actually, unique existence of these minimisers is guaranteed through Assumption \Cref{ass:convexity1}, since $\alpha \mapsto H^R$ and $\alpha_0 \mapsto \HH_0^R$ are strongly convex and continuous. However, we do still need to make the assumption that $\mathring \alpha_0$ does not depend on $\alpha$.
%----------Remark: Our minimisers depend on more arguments
\begin{remark}
In the existing literature, the analogous assumption to \fullref{ass:minimisersofhamiltonians} is stricter. In fact, in \citep[][pp 1545 \& 1547]{carmona2016probabilistic}, they assume the following more stringent assumptions, that:
\begin{itemize}
\item the minimiser $\mathring \alpha$ depends only on $(t, x_0, x, y, \mu)$, and therefore is independent of $\alpha_0$.
\item whereas the minimiser $\mathring \alpha_0$ depends only on $(t, x_0, y_0, \mu)$, and is therefore independent of the second marginal of $\nu$.
\end{itemize} 
% \textcolor{Purple}{Although seemingly subtle, this relaxation of assumptions as compared to the current literature is a major benefit for practitioners, and at the same time, requires significantly sophisticated and new mathematical tools to handle.
% %Note that a critical assumption we make is that the minimiser $\mathring \alpha_0$ does not depend on $\alpha$. However we do allow for $\mathring \alpha$ to depend on $\alpha_0$.  Specifically: 
% Unfortunately, assuming that the minimisers $\mathring \alpha$ and $\mathring \alpha_0$ are independent of of $\alpha_0$ and the second marginal of $\nu$ respectively precludes many useful models, namely models with `meaningful' interactions between the major and representative minor player. Hence, our less restrictive assumption is a significant departure from existing literature. Moreover, our setting intuitively makes sense, as it seems plausible that the optimal strategy of the minor player depends on that of the major player, but not the other way around. The reason such restrictive assumptions are made in the existing literature is due to simplicity, as this added dependency structure creates further complications in the statement of the sufficient stochastic maximum principle. The resulting FBSDE thus requires more sophisticated techniques to tackle, as seen in \Cref{sec:solvabilityofFBSDESMP}. Indeed, it is plain to see that the practical example we choose to investigate in Section ... would not be possible in the setting of \citep{carmona2016probabilistic}.}
Although seemingly subtle, we believe our relaxed assumptions are substantially beneficial for real world applications. Indeed, assuming that the minimisers $\mathring \alpha$ and $\mathring \alpha_0$ are independent of of $\alpha_0$ and the second marginal of $\nu$ respectively precludes many desirable models, namely models with non-trivial interaction between the major and representative minor player. Moreover, our relaxed assumptions intuitively make sense, as it seems plausible that the optimal strategy of the minor player can depend on that of the major player, but not the other way around. Indeed, it is plain to see that the practical examples we choose to investigate in \Cref{sec:ESGapplications} would not be possible to formulate without our relaxed assumptions. 
% On the other hand, the implications of relaxing these assumptions are that significantly sophisticated and new mathematical tools will be required in the subsequent mathematical analysis. Indeed, our added dependency structure creates further complications in the statement of the upcoming sufficient stochastic maximum principle. Namely, the resulting FBSDE requires more sophisticated techniques to tackle, as seen in \Cref{sec:solvabilityofFBSDESMP}(\textcolor{orange}{Have not proven existence.}).
\end{remark}

%----------Thm: Sufficient SMP major/minor player
\begin{theorem}[Sufficient stochastic maximum principle for MmMFG strategy]
\label{thm:sufficientSMPMm}
%Enforce \fullrefs{ass:convexity}{ass:minimisersofhamiltonians}.
Enforce Assumptions \partref{ass:regularity}, \partref{ass:convexity}, and \partref{ass:minimisersofhamiltonians}. Let $\bar \alpha_{0t} := \mathring \alpha_0(t, X_{0t}, P_{0t}, \LL(X_t, P_t))$ and $ \bar \alpha_{t} := \mathring \alpha(t, X_{0t}, X_t, Y_t, \bar \alpha_{0t}, \LL(X_t))$. Consider the FBSDE
\begin{align}
\label{eqn:FBSDEMmsufficient}
\begin{split}
	\dd X_{0t} &= b_0(t, X_{0t}, \bar \alpha_{0t}, \LL(X_t)) \dd t, \\
	\dd X_t &= b(t, \ubar{X}_t,  \bar \alpha_{0t}, \bar \alpha_{t}, \LL(X_t)) \dd t + \sigma(t, X_t) \dd B_t, \\
	\dd P_{0t} &= - \partial_{x_0} H_0(t, \ubar{X}_t, \ubar{P}_t,  Q_t, \bar \alpha_{0t}, \bar \alpha_t, \LL(X_t))\dd t \\
                    &\quad+ Q_{0t} \dd B_t, \\
	\dd P_t &= - \Big [\partial_x H_0(t, \ubar{X}_t, \circ, P_t,  Q_t, \bar \alpha_{0t}, \bar \alpha_t, \LL(X_t)) \\
			&\quad + \tilde \Ex[ \partial_{\mu} H_0(t, \ubar{\tilde X}_t, \ubar{\tilde P}_t, \tilde Q_t, \tilde{\bar \alpha}_{0t}, \tilde{\bar \alpha}_{t}, \LL(X_t))(X_t)] \Big] \dd t\\
			&\quad + Q_t \dd B_t, \\
	\dd Y_t &= - \partial_x H(t, \ubar{X}_t, Y_t, Z_t, \bar \alpha_{0t}, \bar \alpha_t, \LL(X_t)) \dd t \\
			&\quad + Z_t \dd B_t,
\end{split}
\end{align}
with conditions 
\begin{align}
\label{eqn:FBSDEMmsufficientcond}
\begin{split}
	P_{0T} &=  \partial_{x_0} g_0(X_{0T}, \LL( X_T)), \quad P_{T} =  \partial_{\mu} g_0(X_{0T}, \LL( X_T)) (X_T), \\ 
	Y_T &=  \partial_x g(\ubar{X}_T, \LL(X_T)).
\end{split}
\end{align}
Assume the above FBSDE \crefrange{eqn:FBSDEMmsufficient}{eqn:FBSDEMmsufficientcond} has a unique solution denoted by $(\ubar{\hat{X}}, \ubar{\hat{P}}, \hat Y, \ubar{\hat{Q}}, \hat Z)$. Let $\hat \alpha_{0t} := \mathring \alpha_0(t, \hat X_{0t}, \hat P_{0t}, \LL(\hat X_t, \hat P_t))$ and $\hat \alpha_t := \mathring \alpha(t, \hat X_{0t}, \hat X_t, \hat Y_t, \hat \alpha_{0t}, \LL(\hat X_t))$. Then $(\hat \alpha_0, \hat \alpha)$ is a Nash equilibrium for the \hyperref[MmMFGstrategy]{MmMFG strategy}.
\end{theorem}
\begin{proof}
The proof can be carried out by following the ideas from our proof of \Cref{thm:necessarySMPMm}, and then adapting the arguments made in \citep[][Theorem 4.7]{carmona2015forward} as well as \citep[][Appendix]{carmona2016probabilistic} to deduce sufficiency.
% \textcolor{orange}{Insert proof, or maybe just reference the usual sufficient stochastic maximum principle proofs? It should be clear how to obtain this through the proof of the necessary condition, which we gave.}
\end{proof}
\begin{remark}
Establishing wellposedness for the FBSDE \crefrange{eqn:FBSDEMmsufficient}{eqn:FBSDEMmsufficientcond}, which we note is of McKean-Vlasov type, is in general, a very difficult task, and beyond the scope of this article. It is plausible to conjecture that in the linear-quadratic setting, a general solvability result can be achieved similar to that found in \citep[][\S 5]{carmona2016probabilistic}. However, the presence of the argument $\alpha_0$ within the minimiser $\mathring \alpha$ poses additional problems. Instead, we consider particular examples in \Cref{sec:ESGapplications} where solvability is possible, and we leave general wellposedness results of this FBSDE for future work.
\end{remark}
%-------------------------------Approx Nash equilibria----------------------------
%---------------------------------------------------------------------------------
\subsection{Approximate Nash equilibria for the finite player game}
\label{sec:approximatenashequilibria}
Assuming that the \hyperref[MmMFGstrategy]{MmMFG strategy} can be solved, we can determine $\vep_N$-Nash equilibria in the sense of \Cref{defn:epsilonNash} for the finite player game outlined in \Cref{sec:finiteplayersetting} through the sufficient stochastic maximum principle \Cref{thm:sufficientSMPMm}. We will be relatively informal in this section, and more precise details can be found in \citep[][\S 4]{carmona2016probabilistic}.

To be precise, let $\hat \mu_t$ denote the MmMFG equilibrium found in \textbf{Step} \Cref{MmMFG:Step3} of the \hyperref[MmMFGstrategy]{MmMFG strategy}. Enforce all assumptions in \Cref{thm:sufficientSMPMm} so that it holds, and denote the unique solution of the FBSDE \crefrange{eqn:FBSDEMmsufficient}{eqn:FBSDEMmsufficientcond} by $(\ubar{\hat{X}}, \ubar{\hat{P}}, \hat Y, \ubar{\hat{Q}}, \hat Z)$. Assume also that decoupling fields $\theta_P: [0, T] \times \reals^d$ and $\theta_Y : [0, T] \times \reals^{d_0} \times \reals^d$ exist which are Lipschitz continuous in the space variables, uniformly in $t \in [0, T]$, and such that 
\begin{align*}
    P_t = \theta_P(t, X_t), \\
    Y_t = \theta_Y(t, X_{0t}, X_t).
\end{align*}
Consider the SDEs: 
\begin{align*}
    \dd \check X_{0t} &= b_0(t, \check X_{0t}, \check \alpha_{0t}, \bar \mu_{\check{\bs{X}}_t}) \dd t, \\
    \dd \check X_t^i &= b(t, \check X_{0t}, \check X_t^i, \check \alpha_{0t}, \check \alpha_t^i, \bar \mu_{\check{\bs{X}}_t}) \dd t + \sigma(t, \check X_t^i) \dd B_t^i, \quad X_0^i = x_0^i, \quad i = 1, \dots, N,
\end{align*}
where $\check \alpha_{0t} := \mathring \alpha_0(t, \hat X_{0t}, \hat P_{0t}, \hat \mu_t \circ (I_d, \theta_P(t, \cdot)^{-1})$ and $\check \alpha_t^i := \mathring \alpha(t, \hat X_{0t}, \check X_t^i, \theta_Y(t, X_{0t}, \check X_t^i), \check \alpha_{0t}, \hat \mu_t)$. By adapting arguments given in the proof of \citep[][Theorem 4.1]{carmona2016probabilistic}, it can shown that $(\check \alpha_{0t}, (\check \alpha_t^i)_{1 \leq i \leq N})$ forms an $\vep_N$-Nash equilibrium for the $N + 1$ player game. Moreover, under additional regularity assumptions, the speed of convergence can also be described.

\section{Enlarging the state space}
\label{sec:enlargingstatespace}
\noindent
In this section, we consider an extension of our modelling framework. Namely, we invoke the model framework from \Cref{sec:modelframework}, however with one important distinction. Here, we will enlarge the McKean-Vlasov state process of the representative minor player from $X_t$ (originally given in \cref{eqn:stateprocessminorMKV}) to $\chi_t = (X_t; \gamma_t) \in \reals^{d + k}$, where 
\begin{align*}
	\dd X_t &= b(t, X_{0t}, X_t, \gamma_t, \alpha_{0t}, \alpha_t, \LL(X_t, \gamma_t)) \dd t + \sigma(t, X_t, \gamma_t) \dd B_t,  \\
	\dd \gamma_t &= \alpha_t \dd t,
\end{align*} 
which can be compactly written as
\begin{align*}
	\dd \chi_t &= B(t, X_{0t}, X_t, \gamma_t,  \alpha_{0t}, \alpha_t, \LL(X_t, \gamma_t)) \dd t + \Sigma(t, X_t, \gamma_t) \dd B_t
\end{align*} 
with $B(t, x_0, x, \gamma, \alpha_0, \alpha, \lambda) = (b(t, x_0, x, \gamma,  \alpha_0, \alpha, \lambda); \alpha_t)$ and $\Sigma(t, x, \gamma)  = (\sigma(t, x, \gamma); \bs{0}_{k \times m})$. In order to distinguish this setting from the original modelling framework, we use $\lambda$ to denote a generic element of $\PP_2(\reals^{d + k})$ rather than $\mu$. Furthermore, we will enlarge the state process of the representative minor player (originally given in \cref{eqn:stateprocessminor}) with the fixed flow of probability measures $(\lambda_t)_{\tinT} \subseteq \PP_2(\reals^{d + k})$ in a similar way to the McKean-Vlasov case to obtain $\chi^\lambda_t = (X_t^\lambda; \gamma_t)$. As $\alpha_t$ is the strategy of the representative minor player, $\gamma_t$ can be thought of as their integrated strategy. The purpose of enlarging the state process in this way is to include the joint distribution $\LL(X_t, \gamma_t)$ in the measure arguments. This will be useful in \Cref{sec:ESGapplications}, as we do wish to include the law of the representative firm's integrated strategy in our model. This idea is similar to the framework of extended mean field games, where rather than enlarging the state space, one includes $\LL(X_t, \alpha_t)$ (the joint distribution of the representative minor player's state process and strategy) in measure arguments. However, the fixed point step in the framework of extended mean field games is typically more difficult to tackle. Therefore, we opt to enlarge the state space in this manner to enable us to include the joint distribution $\LL(X_t, \gamma_t)$. Despite $\chi_t$ being the state process of the representative minor player, in the following we will prefer to write the instances of $X_t$ and $\gamma_t$ individually in order to emphasise these components. 

In this new setting, the state processes can be described as:
\begin{align}
	\dd X^\lambda_{0t} &= b_0(t, X^\lambda_{0t}, \alpha_{0t},\lambda_t) \dd t, \\
	\dd X^\lambda_t &= b(t, X^\lambda_{0t}, X^\lambda_t, \gamma_t, \alpha_{0t}, \alpha_t, \lambda_t) \dd t + \sigma(t, X^\lambda_t, \gamma_t) \dd B_t, \\
        \dd X_{0t} &= b_0(t, X_{0t}, \alpha_{0t}, \LL(X_t, \gamma_t)) \dd t, \\
	\dd X_t &= b(t, X_{0t}, X_t, \gamma_t, \alpha_{0t}, \alpha_t, \LL(X_t, \gamma_t)) \dd t + \sigma(t, X_t, \gamma_t) \dd B_t, \\
	\dd \gamma_t &= \alpha_t \dd t,
\end{align}
and the utility functions as:
\begin{align*}
	J_0(\alpha_0, \alpha) &= \Ex \left [\int_0^T f_0(t, X_{0t}, \alpha_{0t}, \LL(X_t, \gamma_t)) \dd t + g_0(X_{0T}, \LL(X_T, \gamma_T)) \right ], \\
	J^\lambda(\alpha_0, \alpha) &= \Ex \left [ \int_0^T f(t, X^\lambda_{0t}, X^\lambda_t, \gamma_t, \alpha_{0t}, \alpha_t, \lambda_t) \dd t + g(X^\lambda_{0T}, X^\lambda_T, \gamma_T, \lambda_T) \right ].
\end{align*}
For $(t, x_0, x, \gamma, y_0, y, \grave{y}, \alpha_0, \alpha, \lambda) \in [0, T] \times \reals^{d_0} \times \reals^d \times \reals^k \times \reals^{d_0} \times \reals^d \times \reals^k  \times A_0 \times A \times \PP_2(\reals^{d + k})$, the Hamiltonians are given by: 
\begin{align*}
	H_0(t, x_0, x, \gamma, y_0, y, \grave{y}, z, \alpha_0, \alpha, \lambda) &= b_0(t, x_0, \alpha_0, \lambda) \cdot y_0 + b(t, x_0, x, \gamma, \alpha_0, \alpha, \lambda) \cdot y + \alpha \cdot \grave{y} + \sigma(t, x, \gamma) \cdot z\\
	&\quad + f_0(t, x_0, \alpha_0, \lambda),\\
	H(t, x_0, x, \gamma, y, \grave{y}, z, \alpha_0, \alpha, \lambda) &= b(t, x_0, x, \alpha_0, \alpha, \lambda) \cdot y + \alpha \cdot \grave{y} + \sigma(t, x, \gamma) \cdot z + f(t, x_0, x, \gamma, \alpha_0, \alpha, \lambda),
\end{align*}
and their reduced versions respectively:
\begin{align*}
	H^R_0(t, x_0, x, \gamma, y_0, y, \grave{y}, \alpha_0, \alpha, \lambda) &= b_0(t, x_0, \alpha_0, \lambda) \cdot y_0 + b(t, x_0, x, \gamma, \alpha_0, \alpha, \lambda) \cdot y + \alpha \cdot \grave{y} + f_0(t, x_0, \alpha_0, \lambda),\\
	H^R(t, x_0, x, \gamma, y, \grave{y}, \alpha_0, \alpha, \lambda) &= b(t, x_0, x, \gamma, \alpha_0, \alpha, \lambda) \cdot y + \alpha \cdot \grave{y} + f(t, x_0, x, \gamma, \alpha_0, \alpha, \lambda).
\end{align*}

Following the same steps from \Cref{sec:sufficientSMP}, and noting that $\chi_t$ and $\chi_t^\lambda$ are the representative minor player's state processes, not $X_t$ and $X_t^\lambda$, the FBSDE derived from the sufficient part of the stochastic maximum principle, namely \crefrange{eqn:FBSDEMmsufficient}{eqn:FBSDEMmsufficientcond}, reads as follows:
\begin{align}
\begin{split}
\label{eqn:FBSDEMmsufficientenlarged}
	\dd X_{0t} &= b_0(t, X_{0t}, \bar \alpha_{0t}, \LL(X_t, \gamma_t)) \dd t, \\
	\dd X_t &= b(t, X_{0t}, X_t, \gamma_t, \bar \alpha_{0t}, \bar \alpha_t, \LL(X_t, \gamma_t)) \dd t +                  \sigma(t, X_t, \gamma_t) \dd B_t, \\
	\dd \gamma_t &= \alpha_t \dd t, \\
	\dd P_{0t} &= - \partial_{x_0} H_0(t, X_{0t}, X_t, \gamma_t, P_{0t}, P_t, \grave{P}_t, Q_t, \bar \alpha_{0t}, \bar \alpha_t, \LL(X_t, \gamma_t)) \dd t \\
    &\quad + Q_{0t} \dd B_t, \\
	\dd \begin{pmatrix} P_t \\ \grave{P}_t \end{pmatrix} &= - \Bigg [ \partial_{(x, \gamma)} H_0(t, X_{0t}, X_t, \gamma_t, \circ, P_t, \grave{P}_t, Q_t, \bar \alpha_{0t}, \bar \alpha_t, \LL(X_t, \gamma_t)) \\
	&\quad + \tilde \Ex [\partial_\lambda H_0(t, \tilde X_{0t}, \tilde X_t, \tilde \gamma_t, \tilde P_{0t}, \tilde P_t, \tilde{\grave{P}}_t, \tilde Q_t, \tilde{\bar{\alpha}}_{0t}, \tilde{\bar{\alpha}}_t, \LL(X_t, \gamma_t))(X_t; \gamma_t)] \Bigg ] \dd t \\
        &\quad + \begin{pmatrix}Q_t \\ \grave{Q}_t \end{pmatrix} \dd B_t,\\
	\dd \begin{pmatrix} Y_t \\ \grave{Y}_t \end{pmatrix} &=  - \partial_{(x, \gamma)} H(t, X_{0t}, X_t, \gamma_t, Y_t, \grave{Y}_t, Z_t, \bar \alpha_{0t}, \bar \alpha_t, \LL(X_t, \gamma_t)) \dd t \\
    &\quad + \begin{pmatrix}Z_t \\ \grave{Z}_t \end{pmatrix} \dd B_t,
\end{split}
\end{align}
with conditions
\begin{align}
\label{eqn:FBSDEMmsufficientenlargedcond}
\begin{split}
	P_{0T} &=  \partial_{x_0} g_0(X_{0T}, \LL(X_T, \gamma_T)), \quad \begin{pmatrix} P_T \\ \grave{P}_T \end{pmatrix} =  \partial_{\lambda} g_0(X_{0T}, \LL( X_T, \gamma_T)) (X_T; \gamma_T), \\ 
	\begin{pmatrix} Y_T \\ \grave{Y}_T \end{pmatrix} &=  \partial_{(x, \gamma)} g(\ubar{X}_T, \gamma_T, \LL(X_T, \gamma_T)).
\end{split}
\end{align}
% It is plain to see that the existence result \Cref{thm:sufficientSMPMmFBSDE} still holds in this setting (\textcolor{orange}{still need to do proof}), and thus \crefrange{eqn:FBSDEMmsufficientlinearenlarged}{eqn:FBSDEMmsufficientlinearenlargedcond} has a unique solution denoted by $(\hat{\ubar{X}}, \hat \gamma, \hat{\ubar{P}}, \hat{\grave{P}}, \hat Y, \hat{\grave{Y}}, \hat{\ubar{Q}}, \hat{\grave{Q}}, \hat Z, \hat{\grave{Z}})$

%-----------------Removing the major player's state process-----------------------
%---------------------------------------------------------------------------------
\section{Applications to Green and sustainable Finance}
\label{sec:ESGapplications}
\noindent
In this section, we consider a problem pertaining to Green and sustainable finance inspired by \citep{lavigne2023decarbonization} as well as \citep{dayanikli2024multi}. The problem involves a large number of statistically identical carbon-emitting firms (the minor players) who wish to maximise their stock price, and a regulator (the major player) who wishes for the economy to flourish, whilst at the same time will penalise firms who behave non-green. Of course, we wish to model this setting through the \hyperref[MmMFGstrategy]{MmMFG strategy} and utilise the power of \Cref{thm:sufficientSMPMm}. The representative firm's strategy $\alpha$ corresponds to their emission rate, whereas the regulator's strategy $\alpha_0$ corresponds to a taxation rate. The firm's state process $X$ corresponds to the log of their stock price, which is adjusted due to the regulator's taxation rate. For simplicity, we assume that interest rates are $0$. Importantly, the regulator possesses no state process, and thus can be thought of as a principal (although, not a leader like in a Stackelberg game). Therefore, everything involving $x_0, y_0, X_0, P_0, Q_0$ disappears from the modelling framework. Ergo, in the following we will omit these arguments. To be precise, removing the major player's state process modifies the FBSDE \crefrange{eqn:FBSDEMmsufficientenlarged}{eqn:FBSDEMmsufficientenlargedcond} to the following:
%----Linear version of preceding FBSDE:
% \begin{align}
% \label{eqn:FBSDEMmsufficientlinearenlargednomajor}
% \begin{split}
% 	\dd X_t  &= \left ( b^{(0)}(t) + b^{(2)}(t) (X_t; \gamma_t) + b^{(3)}(t) \bar \alpha_{0t} + b^{(4)}(t) \bar \alpha_t + b^{(5)}(t) \Ex [(X_t; \gamma_t)] \right ) \dd t \\
% 		&\quad + (\sigma^{(0)}(t) + \sigma^{(2)} (t) (X_t; \gamma_t) ) \dd B_t, \\
% 	\dd \gamma_t &= \bar \alpha_t \dd t, \\
% 	\dd \begin{pmatrix} P_t \\ \grave{P}_t \end{pmatrix} &= - \Bigg ( [b^{(2)}(t)]^\top P_t + [\sigma^{(2)} (t)]^\top Q_t \\
% 	&\quad + [b^{(5)}(t)]^\top \Ex[P_t] + \tilde \Ex [ \partial_\lambda f_0(t, \tilde{\bar{\alpha}}_{0t}, \LL(X_t, \gamma_t))(X_t; \gamma_t)] \Bigg ) \dd t + \begin{pmatrix}Q_t \\ \grave{Q}_t \end{pmatrix} \dd B_t,\\
% 	\dd \begin{pmatrix} Y_t \\ \grave{Y}_t \end{pmatrix} &=  - \Big ( [b^{(2)}(t)]^\top Y_t + [\sigma^{(2)}(t)]^\top Z_t + \partial_{(x, \gamma)} f(t, X_t, \gamma_t, \bar{\alpha}_{0t}, \bar{\alpha}_t, \LL(X_t, \gamma_t)) \Big ) \dd t + \begin{pmatrix}Z_t \\ \grave{Z}_t \end{pmatrix} \dd B_t,
% \end{split}
% \end{align}
% with conditions
% \begin{align}
% \label{eqn:FBSDEMmsufficientlinearenlargednomajorcond}
% 	 &\begin{pmatrix} P_T \\ \grave{P}_T \end{pmatrix} =  \partial_{\lambda} g_0(\LL( X_T, \gamma_T)) (X_T; \gamma_T),
% 	\quad &\begin{pmatrix} Y_T \\ \grave{Y}_T \end{pmatrix} &=  \partial_{(x, \gamma)} g(X_T, \gamma_T, \LL(X_T, \gamma_T)).
% \end{align}
%----Non linear version of preceding FBSDE:
\begin{align}
\label{eqn:FBSDEMmsufficientenlargednomajor}
\begin{split}
	\dd X_t  &= b(t, X_t, \gamma_t, \bar \alpha_{0t}, \bar \alpha_t, \LL(X_t, \gamma_t)) \dd t + \sigma(t, X_t, \gamma_t) \dd B_t, \\
	\dd \gamma_t &= \bar \alpha_t \dd t, \\
	\dd \begin{pmatrix} P_t \\ \grave{P}_t \end{pmatrix} &= - \Bigg [ \partial_{(x, \gamma)} H_0(t, X_t, \gamma_t, P_t, \grave{P}_t, \bar \alpha_{0t}, \bar \alpha_t, \LL(X_t, \gamma_t)) \\
	\quad &+ \tilde \Ex [ \partial_\lambda H_0(t, \tilde X_t, \tilde \gamma_t, \tilde P_t, \tilde{\grave{P}}_t, \tilde Q_t, \tilde{\grave{Q}}_t, \tilde{\bar{\alpha}}_{0t}, \tilde{\bar{\alpha}}_t, \LL(X_t, \gamma_t))(X_t; \gamma_t)] \Bigg ] \dd t + \begin{pmatrix}Q_t \\ \grave{Q}_t \end{pmatrix} \dd B_t,\\
	\dd \begin{pmatrix} Y_t \\ \grave{Y}_t \end{pmatrix} &=  - \partial_{(x, \gamma)} H(t, X_t, \gamma_t, Y_t, \grave{Y}_t, Z_t, \grave{Z}_t, \bar \alpha_{0t}, \bar \alpha_t, \LL(X_t, \gamma_t)) \dd t + \begin{pmatrix}Z_t \\ \grave{Z}_t \end{pmatrix} \dd B_t,
\end{split}
\end{align}
with conditions
\begin{align}
\label{eqn:FBSDEMmsufficientenlargednomajorcond}
	 &\begin{pmatrix} P_T \\ \grave{P}_T \end{pmatrix} =  \partial_{\lambda} g_0(\LL( X_T, \gamma_T)) (X_T; \gamma_T),
	\quad &\begin{pmatrix} Y_T \\ \grave{Y}_T \end{pmatrix} &=  \partial_{(x, \gamma)} g(X_T, \gamma_T, \LL(X_T, \gamma_T)).
\end{align}
We recall that the probability measure $\lambda$ serves as a proxy for the joint distribution of $(X_t ; \gamma_t)$. Moreover, we will denote its first and second marginals as $\lambda_x(\cdot) := \lambda(\cdot \times \reals^{d_0})$ and $\lambda_\gamma(\cdot) := \lambda(\reals^d \times \cdot).$ Furthermore, the probability measure $\nu \in \PP_2(\reals^{(d + k) + (d + k)})$ is a proxy for the joint distribution of $(X_t; \gamma_t; P_t; \grave{P}_t)$.

We would like to emphasise that the following examples, despite seeming simple, can not be tackled with the theory developed previously in the literature. Succinctly put, this is because the minimisers of the the upcoming Hamiltonians will depend on arguments which were forbidden in the previous literature.
\subsection{Example 1}
We will start off with a simple example. The coefficients we choose are given by
\begin{align*}
	b(t, x, \gamma, \alpha_0, \alpha, \lambda) &= \alpha - \alpha_0 - \frac{1}{2} \sigma_t^2, 	&\sigma(t,x, \gamma) &= \sigma_t, \\
	f_0(t, \alpha_0, \lambda) &= \frac{1}{2} \alpha_0^2, 		&g_0(\lambda) &= - \bar \lambda_x + \bar \lambda_\gamma, \\
	f(t, x, \gamma, \alpha_0, \alpha, \lambda) &= \frac{1}{2} \alpha^2 \alpha_0^2, 	&g(x, \gamma, \lambda) &= - x.
\end{align*}
The motivation for choosing such coefficients is the following. First, the coefficients $b$ and $\sigma$ are chosen so as to model the log of a stock price as an arithmetic Brownian motion, and it is adjusted due to the taxation rate of the regulator. $f_0$ states that the regulator wishes to reduce running taxation rates. $g_0$ is chosen so that the regulator desires to maximise the overall economy, and also minimise the emissions of the representative firm. $f$ states that the representative firm wishes to minimise their running emissions level, which is weighted by the regulator's taxation rate. $g$ is chosen so that the representative firm wishes to maximise their terminal stock price.  

The reduced Hamiltonians then read as
\begin{align*}
	H_0^R(t, x, \gamma, y, \grave{y}, \alpha_0, \alpha, \lambda) &= (\alpha - \alpha_0 - \frac{1}{2} \sigma_t^2) y + \alpha \grave{y} + \frac{1}{2} \alpha_0^2,\\
	H^R(t, x, \gamma, y, \grave{y}, \alpha_0, \alpha, \lambda) &= (\alpha - \alpha_0 - \frac{1}{2} \sigma_t^2)y + \alpha \grave{y} + \frac{1}{2} \alpha^2 \alpha_0^2.
\end{align*}
By appealing to the sufficient stochastic maximum principle for the \hyperref[MmMFGstrategy]{MmMFG strategy} \Cref{thm:sufficientSMPMm}, our first task is to minimise $\alpha_0 \mapsto \Ex \left [ H^R_0(t, X_t, \gamma_t, P_t, \grave{P}_t, \alpha_0, \alpha, \LL(X_t, \gamma_t))\right ]$ and $\alpha \mapsto H^R(t, x, \gamma, y, \grave{y}, \alpha_0, \alpha, \lambda)$. This yields the minimisers: 
\begin{align*}
	\mathring \alpha_0(t, \LL(X_t, \gamma_t, P_t, \grave{P}_t)) &= \Ex[P_t], \\
	\mathring \alpha(t, x, \gamma, y, \grave{y}, \alpha_0, \lambda) &= -(y + \grave{y})/ \alpha_0^2. 
\end{align*}
The FBSDE \crefrange{eqn:FBSDEMmsufficientenlargednomajor}{eqn:FBSDEMmsufficientenlargednomajorcond} thus becomes: 
\begin{align*}
	\dd X_t &= \left ( - \frac{Y_t + \grave{Y}_t}{(\Ex[P_t])^2} - \Ex[P_t] - \frac{1}{2} \sigma_t^2 \right ) \dd t + \sigma_t \dd B_t, \\
	\dd \gamma_t &= - \frac{Y_t + \grave{Y}_t}{\Ex[P_t]} \dd t, \\
	\dd \begin{pmatrix} P_t \\ \grave{P}_t \end{pmatrix} &=\begin{pmatrix}Q_t \\ \grave{Q}_t \end{pmatrix} \dd B_t,  \\
	\dd \begin{pmatrix} Y_t \\ \grave{Y}_t \end{pmatrix} &= \begin{pmatrix}Z_t \\ \grave{Z}_t \end{pmatrix} \dd B_t,
\end{align*}
with terminal conditions:
\begin{align*}
&\begin{pmatrix} P_T \\ \grave{P}_T \end{pmatrix} =  \begin{pmatrix} -1 \\  1 \end{pmatrix}, \quad &\begin{pmatrix} Y_T \\ \grave{Y}_T \end{pmatrix} &=  \begin{pmatrix} -1 \\  0 \end{pmatrix}.
\end{align*}
 If the above FBSDE has a solution $(\hat X, \hat \gamma, \hat P, \hat{\grave{P}}, \hat Y, \hat{\grave{Y}}, \hat Q, \hat{\grave{Q}}, \hat Z, \hat{\grave{Z}})$, then the optimal strategy is given by $\hat \alpha_{0t} = \Ex[\hat P_t]$ and $\hat \alpha_t = -(\hat Y_t + \hat{\grave{Y}}_t)/ (\hat \alpha_{0t})^2$. It is plain to see that $\hat P_t = -1, \hat{\grave{P}}_t = 1, \hat Y_t = -1, \hat{\grave{Y}}_t = 0$. Therefore, the optimal strategy is $\hat \alpha_{0t} = -1$ and $\hat \alpha_t = 1$.

\subsection{Example 2}
One reason for the simplistic solution of the previous example stems from the fact that the major player's running cost $f_0$ is essentially independent of the representative minor player. Therefore, to obtain a framework which is more interesting, an option is to adjust the running cost of the major player $f_0$. The coefficients we now choose are given by
\begin{align*}
	b(t, x, \gamma, \alpha_0, \alpha, \lambda) &= \alpha - \alpha_0 - \frac{1}{2} \sigma_t^2, 	&\sigma(t,x, \gamma) &= \sigma_t, \\
	f_0(t, \alpha_0, \lambda) &= \frac{1}{2} \alpha_0^2/ \bar\lambda_\gamma, 		&g_0(\lambda) &= - \bar \lambda_x + \bar \lambda_\gamma, \\
	f(t, x, \gamma, \alpha_0, \alpha, \lambda) &= \frac{1}{2} \alpha^2 \alpha_0^2, 	&g(x, \gamma, \lambda) &= - x.% \frac{1}{2} x^2.
\end{align*}
The interpretation in this modification of $f_0$ is that the regulator punishes firms with low emissions less than firms with high emissions. The reduced Hamiltonians then read as
\begin{align*}
	H_0^R(t, x, \gamma, y, \grave{y}, \alpha_0, \alpha, \lambda) &= (\alpha - \alpha_0 - \frac{1}{2} \sigma_t^2) y + \alpha \grave{y} + \frac{1}{2} \alpha_0^2/\bar \lambda_\gamma,\\
	H^R(t, x, \gamma, y, \grave{y}, \alpha_0, \alpha, \lambda) &= (\alpha - \alpha_0 - \frac{1}{2} \sigma_t^2)y + \alpha \grave{y} + \frac{1}{2} \alpha^2 \alpha_0^2.
\end{align*}
Minimising $\alpha_0 \mapsto \Ex \left [ H^R_0(t, X_t, \gamma_t, P_t, \grave{P}_t, \alpha_0, \alpha, \LL(X_t, \gamma_t))\right ]$ and $\alpha \mapsto H^R(t, x, \gamma, y, \grave{y}, \alpha_0, \alpha, \lambda)$ yields the minimisers: 
\begin{align*}
	\mathring \alpha_0(t, \LL(X_t, \gamma_t, P_t, \grave{P}_t)) &= \Ex[P_t]\Ex[\gamma_t], \\
	\mathring \alpha(t, x, \gamma, y, \grave{y}, \alpha_0, \lambda) &= -(y + \grave{y})/\alpha_0^2. 
\end{align*}
The FBSDE \crefrange{eqn:FBSDEMmsufficientenlargednomajor}{eqn:FBSDEMmsufficientenlargednomajorcond} thus becomes: 
\begin{align*}
	\dd X_t &= \left ( - \frac{Y_t + \grave{Y}_t}{(\Ex[P_t])^2 (\Ex[\gamma_t])^2} - \Ex[P_t]\Ex[\gamma_t] - \frac{1}{2} \sigma_t^2 \right ) \dd t + \sigma_t \dd B_t, \\
	\dd \gamma_t &= - \frac{Y_t + \grave{Y}_t}{(\Ex[P_t])^2 (\Ex[\gamma_t])^2} \dd t, \\
	\dd \begin{pmatrix} P_t \\ \grave{P}_t \end{pmatrix} &= \left (0; \frac{1}{2}(\Ex[P_t])^2(\Ex[\gamma_t])^2 \Ex[\partial_{\lambda_\gamma} (1/ \Ex[\gamma_t])(\gamma_t)]\right) \dd t +  \begin{pmatrix}Q_t \\ \grave{Q}_t \end{pmatrix} \dd B_t,  \\
	\dd \begin{pmatrix} Y_t \\ \grave{Y}_t \end{pmatrix} &= \begin{pmatrix}Z_t \\ \grave{Z}_t \end{pmatrix} \dd B_t,
\end{align*}
with terminal conditions:
\begin{align*}
&\begin{pmatrix} P_T \\ \grave{P}_T \end{pmatrix} =  \begin{pmatrix} -1 \\  1 \end{pmatrix}, \quad &\begin{pmatrix} Y_T \\ \grave{Y}_T \end{pmatrix} &=  \begin{pmatrix} -1 \\  0 \end{pmatrix}.
\end{align*}
 If the above FBSDE has a solution $(\hat X, \hat \gamma, \hat P, \hat{\grave{P}}, \hat Y, \hat{\grave{Y}}, \hat Q, \hat{\grave{Q}}, \hat Z, \hat{\grave{Z}})$, then the optimal strategy is given by $\hat \alpha_{0t} = \Ex[\hat P_t]\Ex[\hat \gamma_t]$ and $\hat \alpha_t = -(\hat Y_t + \hat{\grave{Y}}_t)/ (\hat \alpha_{0t})^2$.

 First, we immediately have that $\hat P_t = -1, \hat Y_t = -1, \hat{\grave{Y}}_t = 0$. This suggests that we study the ODE 
 \begin{align*}
    \dd (\Ex[\gamma_t]) = 1/(\Ex[\gamma_t])^2 \dd t, \quad \Ex[\gamma_0] = 0.
 \end{align*}
 This yields the real solution of $\Ex[\hat \gamma_t] = 3^{1/3} t^{1/3}$. Therefore the optimal strategy is $\hat \alpha_{0t} = -3^{1/3} t^{1/3}$ and $\hat \alpha_t = 1/(3^{2/3} t^{2/3})$.

%--------------------------------ESG example 3
\subsection{Example 3}

Although Example 2 is more interesting than Example 1, we still have not been able to achieve substantial interaction between the major and minor player. Again our strategy is to modify the coefficient $f_0$, so that the coefficients we choose are given by:
\begin{align*}
	b(t, x, \gamma, \alpha_0, \alpha, \lambda) &= \alpha - \alpha_0 - \frac{1}{2} \sigma_t^2, 	&\sigma(t,x, \gamma) &= \sigma_t, \\
	f_0(t, \alpha_0, \lambda) &= \frac{1}{2} \alpha_0^2/ \bar\lambda_\gamma + \alpha_0 \bar \lambda_x, 		&g_0(\lambda) &= - \bar \lambda_x + \bar \lambda_\gamma, \\
	f(t, x, \gamma, \alpha_0, \alpha, \lambda) &= \frac{1}{2} \alpha^2 \alpha_0^2, 	&g(x, \gamma, \lambda) &= - x.% \frac{1}{2} x^2.
\end{align*}
The interpretation in this modification of $f_0$ is still that the regulator punishes firms with low emissions less than firms with high emissions, but in addition, the regulator's taxation rate is affected by the overall performance of the economy. 

The reduced Hamiltonians then read as
\begin{align*}
	H_0^R(t, x, \gamma, y, \grave{y}, \alpha_0, \alpha, \lambda) &= (\alpha - \alpha_0 - \frac{1}{2} \sigma_t^2) y + \alpha \grave{y} + \frac{1}{2} \alpha_0^2/\bar \lambda_\gamma + \alpha_0 \bar \lambda_x,\\
	H^R(t, x, \gamma, y, \grave{y}, \alpha_0, \alpha, \lambda) &= (\alpha - \alpha_0 - \frac{1}{2} \sigma_t^2)y + \alpha \grave{y} + \frac{1}{2} \alpha^2 \alpha_0^2.
\end{align*}
Minimising $\alpha_0 \mapsto \Ex \left [ H^R_0(t, X_t, \gamma_t, P_t, \grave{P}_t, \alpha_0, \alpha, \LL(X_t, \gamma_t))\right ]$ and $\alpha \mapsto H^R(t, x, \gamma, y, \grave{y}, \alpha_0, \alpha, \lambda)$ yields the minimisers: 
\begin{align*}
	\mathring \alpha_0(t, \LL(X_t, \gamma_t, P_t, \grave{P}_t)) &= (\Ex[P_t] - \Ex[X_t])\Ex[\gamma_t], \\
	\mathring \alpha(t, x, \gamma, y, \grave{y}, \alpha_0, \lambda) &= -(y + \grave{y})/\alpha_0^2. 
\end{align*}
The FBSDE \crefrange{eqn:FBSDEMmsufficientenlargednomajor}{eqn:FBSDEMmsufficientenlargednomajorcond} thus becomes: 
\begin{align*}
	\dd X_t &= \left ( - \frac{Y_t + \grave{Y}_t}{(\Ex[P_t] - \Ex[X_t])^2 (\Ex[\gamma_t])^2} - (\Ex[P_t] - \Ex[X_t])\Ex[\gamma_t] - \frac{1}{2} \sigma_t^2 \right ) \dd t + \sigma_t \dd B_t, \\
	\dd \gamma_t &= - \frac{Y_t + \grave{Y}_t}{(\Ex[P_t] - \Ex[X_t])^2 (\Ex[\gamma_t])^2} \dd t, \\
	\dd \begin{pmatrix} P_t \\ \grave{P}_t \end{pmatrix} &= \left ((\Ex[P_t] - \Ex[X_t])\Ex[\gamma_t]; \frac{1}{2}(\Ex[P_t] - \Ex[X_t])^2(\Ex[\gamma_t])^2 \Ex[\partial_{\lambda_\gamma} (1/ \Ex[\gamma_t])(\gamma_t)]\right) \dd t +  \begin{pmatrix}Q_t \\ \grave{Q}_t \end{pmatrix} \dd B_t,  \\
	\dd \begin{pmatrix} Y_t \\ \grave{Y}_t \end{pmatrix} &= \begin{pmatrix}Z_t \\ \grave{Z}_t \end{pmatrix} \dd B_t,
\end{align*}
with terminal conditions:
\begin{align*}
&\begin{pmatrix} P_T \\ \grave{P}_T \end{pmatrix} =  \begin{pmatrix} -1 \\  1 \end{pmatrix}, \quad &\begin{pmatrix} Y_T \\ \grave{Y}_T \end{pmatrix} &=  \begin{pmatrix} -1 \\  0 \end{pmatrix}.
\end{align*}
 If the above FBSDE has a solution $(\hat X, \hat \gamma, \hat P, \hat{\grave{P}}, \hat Y, \hat{\grave{Y}}, \hat Q, \hat{\grave{Q}}, \hat Z, \hat{\grave{Z}})$, then the optimal strategy is given by $\hat \alpha_{0t} = (\Ex[\hat P_t] + \Ex[\hat X_t])\Ex[\hat \gamma_t]$ and $\hat \alpha_t = -(\hat Y_t + \hat{\grave{Y}}_t)/ (\hat \alpha_{0t})^2$. 
 
 Unfortunately we were unable to solve the above system explicitly. In fact more generally, it seems probable that the resulting FBSDE from pushing the complexity of the system further than this will not be explicitly solvable. However, we should highlight that the formulation of the preceding FBSDE was only possible due the sufficient stochastic maximum principle developed in this article. In future work, we would like to study such systems using numerical methods.

% %-----------------------------Numerical Analysis----------------------------------
% %---------------------------------------------------------------------------------
% \section{Numerical analysis}
% \noindent
% \textcolor{Purple}{??}

%---------------------------------Conclusion--------------------------------------
%---------------------------------------------------------------------------------
\section{Conclusion}
\noindent
In this article, we have considered a mean field model for a game with a continuum of carbon-emitting firms, who wish to maximise their stock performance, and a regulator who wishes for the economy to flourish whilst simultaneously punishing firms who behave non-green. Our desire was to model this through the theory of mean field games with major and minor players, however we found that with the existing results in the literature, non-trivial models were not feasible. By considering the scenario where the minimisers of the major and representative minor players' Hamiltonians depend on certain additional arguments, we extended the necessary and sufficient stochastic maximum principle developed in \citep{carmona2016probabilistic}. Using this theorem, we were able to solve some examples explicitly pertaining to problems in Green and sustainable finance. Future work would be to consider more complex models which fit within our framework and where explicit solutions may not exist, therefore requiring the use of numerical methods. We would also like to formulate our problem through a mean field Stackelberg game, and compare the results to the ones achieved through our mean field game with major and minor players.

%---------------------------------Acknowledgements--------------------------------
%---------------------------------------------------------------------------------
\section*{Acknowledgements}
\noindent
Anna Aksamit, Kaustav Das, Ivan Guo, Kihun Nam, and Zhou Zhou have been supported by the Australian Research Council (Grant DP220103106).

% \section*{Conflict of Interest}
% \noindent The authors report there are no competing interests to declare.

%------------------------------References-----------------------------------------
%---------------------------------------------------------------------------------
\renewcommand{\bibname}{References}
\bibliographystyle{plainnat}
\bibliography{refsdp}

\appendix

%------------------------------L derivatives--------------------------------------
%---------------------------------------------------------------------------------
\section{L differentiation}
\label{appen:Lderivative}
\noindent
%------------------------Differentiation of functions of measures
In this appendix, we will provide a brief description of the differential calculus on the space of probability measures $\PP_2(\reals^q)$ that we extensively utilise in this article. There are a few notions for this, but the one that is appropriate for our needs is called L differentiability, introduced by Lions in his lectures at the Coll\`ege de France, see \citep{cardaliaguet2010notes} for a comprehensive review of these lectures. Most of this appendix is an informal summary based off \citep[][\S 5.2]{carmona2018mfgI} as well as \citep[][\S 3.1]{carmona2015forward}, and we refer the reader here for more in depth discussions details, especially regarding technical results. Moreover, we will extensively be utilising the matrix notation from \Cref{remark:matrixnotation}.
%------------------------------L derivative defns----------------------------------
%----------------------------------------------------------------------------------
\subsection{Definitions}
\label{appen:Lderivativedefns}
To our knowledge, the definition of L differentiability in the current literature is given for functions which map $\PP_2(\reals^q)$ to $\reals$. However, L differentiation is often utilised in the context where the codomain is $\reals^p$ for some $p \geq 1$, despite the definition not being stated explicitly in the literature. Therefore, our treatment of L differentiation will cover this more general case. This is quite straightforward and just requires one to consider the function outputs component wise, and then apply the rules of L differentiation to the individual components. In fact, this is probably why this more general setting has not been explicitly treated in the literature. Despite this, we feel for completeness, clarity, and convenience to the reader, that this general setting should be addressed explicitly. One can also derive similar definitions for functions which map $\PP_2(\reals^q)$ to $\reals^{p_1 \times p_2}$ (which would be useful for example, when the diffusion coefficient $\sigma$ in \cref{eqn:stateprocessminor} depends on $\mu$); we leave this adaptation to the reader.

The definition of L derivative requires the notion of lifting functions of probability measures to spaces of random variables. In order for lifting to be well-defined, we require the following surjection type result: For every $\mu \in \PP_2(\reals^q)$, there exists a probability space $(\tilde \Omega, \tilde \FF, \tilde \Pro)$ and square integrable  $\reals^q$ valued random variable $X$ on it such that $\tilde \Pro \circ X^{-1} = \mu$. 
\begin{definition}[Lifted function]
Let $ \PP_2(\reals^q) \ni \mu \mapsto \phi(\mu) \in \reals^p$. The lifting of $\phi$ at $\mu \in \PP_2(\reals^q)$ is denoted by $\tilde \phi$ and defined as $L^2(\tilde \Omega, \tilde \FF, \tilde \Pro ; \reals^q) \ni X \mapsto \tilde \phi(X):= \phi(\LL(X)) \in \reals^p$, where $(\tilde \Omega, \tilde \FF, \tilde \Pro)$ is a probability space and $X$ is a square integrable random variable on it such that $\tilde \Pro \circ X^{-1} = \mu$.
\end{definition}

\begin{definition}[L derivative]
Let $ \PP_2(\reals^q) \ni \mu \mapsto \phi(\mu) \in \reals^p$ and denote by $\tilde \phi$ its lifting. Fix $\mu \in  \PP_2(\reals^q)$ and let $X, H \in L^2(\tilde \Omega, \tilde \FF, \tilde \Pro; \reals^q)$ such that $\tilde \Pro \circ X^{-1} = \mu$. Then $\phi$ is $L$ differentiable at $\mu$ if and only if the Fr\'echet derivatives of the components of $\tilde \phi \equiv (\tilde \phi_1, \dots, \tilde \phi_p)$ at $X$, denoted by $D \tilde \phi_i(X)$ and identified with random variables in $L^2(\tilde \Omega, \tilde \FF, \tilde \Pro ; \reals^q)$ through Riesz representation, exist. The $L$ derivative of $\phi$ at $\mu$ is defined as $D \tilde \phi(X) := (D \tilde \phi_1(X), \dots, D \tilde \phi_p(X))^\top \in L^2(\tilde \Omega, \tilde \FF, \tilde \Pro; \reals^{p \times q})$, which is a matrix valued random variable.
\end{definition}
To further elaborate on the above definitions, first note that the Fr\'echet derivatives of the components $D \tilde \phi_i(X)$ satisfy
\begin{align*}
	\tilde \phi_i(X + H) %&= \tilde \phi_i(X) + D \tilde \phi_i(X) [H] + \littleo{H} \\
	&= \tilde \phi_i(X) + \Ex[D \tilde \phi_i(X) \cdot H] + \littleo{H}, \quad 1 \leq i \leq p.
\end{align*}
Therefore, 
\begin{align*}
	\tilde \phi(X + H) = \tilde \phi(X) + \Ex[D \tilde \phi (X) H] + \littleo{H}.
\end{align*} 
\begin{remark}
The distribution of the L derivative does not depend on the random variable used to represent it. That is, if $X, \tilde X$ are two $\reals^q$ valued random variables possibly defined on different probability spaces $(\Omega, \FF, \Pro), (\tilde \Omega, \tilde \FF, \tilde \Pro)$ such that $\Pro \circ X^{-1} = \tilde \Pro \circ \tilde X^{-1} = \mu$, then $D \tilde \phi(X) \stackrel{d}{=} D \tilde \phi(\tilde X)$. Therefore, one can identify the L derivative with a Borel measurable function which we denote by $\partial_\mu \phi(\mu)(\cdot): \reals^q \to \reals^{p \times q}$, so that $D \tilde \phi (X) = \partial_\mu \phi(\mu)(X)$ whenever $X$ is a $\reals^q$ valued square integrable random variable supported on some probability space and such that $\LL(X) = \mu$. An illustrative example is the following. Consider the function 
\begin{align*}
	\phi(\mu) = \int_{\reals^q} h(x) \mu( \dd x)
\end{align*}
where $h: \reals^q \to \reals^p$, and the preceding integral being a Bochner integral. Therefore, $\phi: \PP_2(\reals^q) \to \reals^p$. Hence, $\partial_\mu \phi(\mu)(\cdot): \reals^q \to \reals^{p + q}$, and it can be shown that $\partial_\mu \phi(\mu)(\cdot) = \partial_x h(\cdot)$ so that the L derivative is identified with the Jacobian of $h$. Additionally, when $h(x) = x$, then $\partial_\mu \phi(\mu)(\cdot) = I_{p \times q}$.
\end{remark}

\begin{definition}[L convexity]
Let $\PP_2(\reals^q) \ni \mu \mapsto \phi(\mu) \in \reals$. Then $\phi$ is convex if for each $\mu, \mu' \in \PP_2(\reals^q)$, we have 
\begin{align*}
	\phi(\mu') - \phi(\mu) \geq \tilde \Ex [ \partial_\mu \phi(\mu)(\tilde X) \cdot (\tilde X' - \tilde X)]
\end{align*}
whenever $\tilde X, \tilde X' \in L^2(\tilde \Omega, \tilde \FF, \tilde \Pro)$ such that $\tilde X \sim \mu$ and $\tilde X' \sim \mu'$.
\end{definition}

\subsection{Calculating L derivatives in some particularly useful settings}
In even the slightest non-trivial settings, calculating L derivatives can be nonintuitive. In this section we will consider a few examples which should help clarify how some calculations in this article are performed.

In many instances of L differentiation, the fact that the distribution of the L derivative is independent of the random variable and probability space used to represent it can be exploited in a convenient way. For example, consider the case where we have a function $\reals^d \times \PP_2(\reals^d) \ni (x, \mu) \mapsto \psi(x, \mu) \in \reals^d$. By exploiting the fundamental theorem of Fr\'echet differentiation, and utilising so-called partial L differentiation, we have the following representations
\begin{align*}
	\psi(X + Y, \LL(X + Y)) &= \psi(X, \LL(X)) + \int_0^1 \Big ( \partial_x \psi(X + \lambda Y, \LL(X + \lambda Y)) \cdot Y \\
	&\quad + \tilde \Ex[ \partial_\mu \psi(X + \lambda Y, \LL(X + \lambda Y))(\tilde X + \lambda \tilde Y) \cdot \tilde Y] \Big ) \dd \lambda, \\
	\Ex[ \psi(X + Y, \LL(X + Y))] &= \Ex[ \psi(X, \LL(X))] + \int_0^1 \Big ( \Ex [ \partial_x \psi(X + \lambda Y, \LL(X + \lambda Y)) \cdot Y ] \\
	&\quad + \Ex \tilde \Ex [ \partial_\mu \psi(X + \lambda Y, \LL(X + \lambda Y))(\tilde X + \lambda \tilde Y) \cdot \tilde Y] \Big ) \dd \lambda.
\end{align*}
Here random variables with tildes denote independent copies of the respective original versions defined on a copy of the respective original probability space. For example, in the above $X, \tilde X$ are defined on probability spaces $(\Omega, \FF, \Pro), (\tilde \Omega, \tilde \FF, \tilde \Pro)$ respectively such that $\Pro \circ X^{-1} = \tilde \Pro \circ \tilde X^{-1}$. This technique of copying the random variables and the underlying probability space is crucial. For example:
\begin{align*}
	\psi(X, \LL(X + Y)) - \psi(X, \LL(X)) = \int_0^1 \tilde \Ex[ \partial_\mu \psi(X, \LL(X + \lambda Y))(\tilde X + \lambda \tilde Y) \cdot \tilde Y] \Big ) \dd \lambda.
\end{align*}
The above LHS is a random variable. On the other hand, if we did not utilise the tilde notation on the RHS, then it would be non-random, which does not make sense.

%------- Partial component L differentiation
Sometimes it will be necessary to take L derivative of a function in a marginal of the probability measure being mapped. To be more precise, consider the situation where we have the map $\PP_2(\reals^{q_1 + q_2}) \ni \nu \mapsto \phi(\nu) \in \reals^p$. Let $\mu_i, i = 1, 2$, denote the marginals of $\nu$ in the sense that $\mu_1(\cdot) = \nu(\cdot \times \reals^{q_2})$ and $\mu_2(\cdot) = \nu(\reals^{q_1} \times \cdot)$. Consider the setting when $\phi(\nu) = \check \phi(\mu_1)$, where $\PP_2(\reals^{q_1}) \ni \mu_1 \mapsto \check \phi(\mu_1) \in \reals^p$. Our objective is to express $\reals^{q_1 + q_2} \ni (x, y) \mapsto \partial_\nu \phi(\nu)(x, y) \in \reals^{p \times (q_1 + q_2)}$ in terms of $\reals^{q_1} \ni x \mapsto \partial_{\mu_1} \check \phi(\mu_1)(x) \in \reals^{p \times q_1}$. It not too difficult to show that 
\begin{align*}
	\partial_\nu \phi(\nu)(x, y) = (\partial_{\mu_1} \check  \phi (\mu_1)(x), \bs{0}_{p \times q_2}).
\end{align*}
For illustration, consider the example where
\begin{align*}
	\phi(\nu) = \check  \phi(\mu_1) = \int_{\reals^{d_1}} h(x) \mu_1(\dd x)
\end{align*}
with $h : \reals^{q_1} \to \reals^{p}$. Then we have 
\begin{align*}
	\partial_\nu \phi(\nu)(x, y) =
\begin{pmatrix}
	\partial_{x_1} h_1(x)& \cdots& \partial_{x_{q_1}} h_1(x)& \bs{0}_{1 \times q_2} \\
	\vdots 				& \ddots & \vdots& \vdots \\
	\partial_{x_1} h_{p}(x)& \cdots& \partial_{x_{q_1}} h_{p}(x)& \bs{0}_{1 \times q_2} \\
\end{pmatrix}
\end{align*}
In other words, $\partial_\nu \phi(\nu)(x, y) = ((\partial_x h_1(x), \bs{0}_{1 \times q_1}); \dots; (\partial_x h_{p}(x), \bs{0}_{1 \times q_2}))$.

%--------------------------Necessary SMP proof------------------------------------
%---------------------------------------------------------------------------------
\section{Proof of \Cref{thm:necessarySMPMm}}
\label{appen:necessarySMPproof}
\noindent
The proof of the necessary stochastic maximum principle will be undertaken over the next few subsections. We will isolate the cases for the major player and representative minor player. Our proof is inspired by arguments made in \citep[][\S 4.1]{carmona2015forward}, and we refer the reader here for any details that may not pop out as obvious. The main purpose of including the proof is to illustrate the asymmetry that the necessary condition implies in the minimisations of the major and representative minor player's Hamiltonians. Recall that we write $\ubar{x} := (x_0, x)$, and $\circ$ in place of any arguments of functions that are redundant. In the following we will write $(\hat \alpha_0, \hat \alpha) \equiv (\hat \alpha^\mu_0, \hat \alpha^\mu)$ to denote the Nash equilibrium obtained from \textbf{Step} \Cref{MmMFG:Step2} of the \hyperref[MmMFGstrategy]{MmMFG strategy}, thereby suppressing the explicit dependence on $\mu$.

% {\color{orange}
% For convenience for us, I have inserted \textbf{Step 2.} of the MmMFG strategy here (and will remove it for the actual paper): 
% \begin{enumerate}
% \item[\textbf{Step 2.}]
% Find a Nash equilibrium $(\hat \alpha_0^{\mu}, \hat \alpha^{\mu})$ to the following stochastic differential game:
% \begin{align*}
% 	J_0(\alpha_0, \alpha) &= \Ex \left [\int_0^T f_0(t, X_{0t}, \alpha_{0t}, \LL(X_t)) \dd t + g_0(X_{0T}, \LL(X_T)) \right ], \\
% 	J^{\mu}(\alpha_0, \alpha) &= \Ex \left [ \int_0^T f (t, X_{0t}^{\mu}, X_t^{\mu}, \alpha_{0t}, \alpha_{t}, \mu_t) \dd t + g(X_{0T}^{\mu}, X_T^{\mu}, \mu_T) \right ],
% \end{align*}
% where we minimise $J_0$ over $\alpha_0 \in \AAbb_0$ and $J^{\mu}$ over $\alpha \in \AAbb$ subject to the following state process(es) of the major and representative minor player:
% \begin{align}
% 	\dd X_{0t}^{\mu} &= b_0(t, X_{0t}^{\mu}, \alpha_{0t}, \mu_t) \dd t, \label{eqn:stateprocessmajor} \\
% 	\dd X^{\mu}_t &= b(t, X_{0t}^{\mu}, X^{\mu}_t, \alpha_{0t}, \alpha_{t}, \mu_t) \dd t + \sigma(t, X^{\mu}_t) \dd B_t, \label{eqn:stateprocessminor}
% \end{align}
% as well as their McKean-Vlasov versions
% \begin{align}
% 	\dd X_{0t} &= b_0(t, X_{0t}, \alpha_{0t}, \LL(X_t)) \dd t, \label{eqn:stateprocessmajorMKV} \\
% 	\dd X_t &= b(t, X_{0t}, X_t, \alpha_{0t}, \alpha_{t}, \LL(X_t)) \dd t + \sigma(t, X_t) \dd B_t, \label{eqn:stateprocessminorMKV}
% \end{align}
% \end{enumerate}
% }

%-------------------------------The major player----------------------------------
%---------------------------------------------------------------------------------
\subsection{The major player's problem}
We first address the optimisation for the major player. As this is the major player's problem, we are considering the McKean-Vlasov versions of the state processes, namely \crefrange{eqn:stateprocessmajorMKV}{eqn:stateprocessminorMKV}.

Consider a perturbation of their optimal strategy in the direction $\beta_0 \in \AAbb_0$ given by $\alpha_0^{\vep_0} := \hat \alpha_0 + \vep_0 \beta_0$, which for a sufficient small $\vep_0$ remains in $\AAbb_0$. Denote by $\hat X_0$ and $X_0^{\vep_0}$ the major player's state process \cref{eqn:stateprocessmajorMKV} along the strategies $(\hat \alpha_0, \hat \alpha)$ and $(\alpha_0^{\vep_0}, \hat \alpha)$ respectively. Similarly, denote by $\hat X$ and $\hat X^{\vep_0}$ for the minor player's state process \cref{eqn:stateprocessminorMKV} along the strategies $(\hat \alpha_0, \hat \alpha)$ and $(\alpha_0^{\vep_0}, \hat \alpha)$ respectively.

The main idea is to consider the G\^ateaux derivative of $\alpha_0 \mapsto J_0(\alpha_0, \hat \alpha)$ at $\hat \alpha_0$ in the direction $\beta_0$, given by
\begin{align}
\label{eqn:gateauxJ0i}
\begin{split}
	&\lim_{\vep_0 \downarrow 0} \vep_0^{-1} \left [J_0(\alpha_0^{\vep_0}, \hat \alpha) - J_0(\hat \alpha_0, \hat \alpha) \right ] \\
	&= \lim_{\vep_0 \downarrow 0} \vep_0^{-1} \Ex \Bigg[ \int_0^T \left (f_0(t, X_{0t}^{\vep_0}, \alpha^{\vep_0}_{0t}, \LL(X_t^{\vep_0})) - f_0(t, \hat X_{0t}, \hat \alpha_{0t}, \LL(\hat X_t)) \right ) \dd t \\&\quad + g_0(X_{0T}^{\vep_0}, \LL(X_T^{\vep_0})) - g_0(\hat X_{0T}, \LL(\hat X_T)) \Bigg].
\end{split}
\end{align}
\begin{enumerate}[label = Step M(\roman*)]
\item
We now define processes $V_0$ and $V$ (taking values in $\reals^{d_0}$ and $\reals^d$ respectively) through the SDEs:
\begin{align*}
	\dd V_{0t} &= \Bigg (\partial_{x_0} b_0(t, X_{0t}, \alpha_{0t}, \LL(X_t)) V_{0t} + \partial_{\alpha_0} b_0(t, X_{0t}, \alpha_{0t}, \LL(X_t)) \beta_{0t} \\
	\quad &+ \tilde \Ex[\partial_{\mu} b_0(t, X_{0t}, \alpha_{0t}, \LL(X_t))(\tilde X_t) \tilde V_{t} ] \Bigg ) \dd t \\
	V_{00} &= 0,
\end{align*}
and
\begin{align*}
	\dd V_{t} &= \Bigg (\partial_{x_0} b(t, X_{0t}, X_t, \alpha_{0t}, \alpha_t, \LL(X_t)) V_{0t} + \partial_x b(t, X_{0t}, X_t, \alpha_{0t}, \alpha_{t}, \LL(X_t)) V_{t} \\ &+ \partial_{\alpha_0} b(t, X_{0t}, X_t, \alpha_{0t}, \alpha_{t}, \LL(X_t)) \beta_{0t} + \tilde \Ex[\partial_{\mu} b(t, X_{0t}, X_t, \alpha_{0t}, \alpha_{t}, \LL(X_t))(\tilde X_t) \tilde V_{t} ] \Bigg ) \dd t \\
	&\quad + (\partial_x \sigma(t, X_t)) V_{t} \dd B_t, \\
	V_{0} &= 0,
\end{align*}
respectively. Here we interpret $\partial_x \sigma(t, X_t) \in \reals^{(d \times m) \times d}$. $V_0$ and $V$ serve in a sense as derivatives of $\vep_0 \mapsto X_0^{\vep_0}$ and $\vep_0 \mapsto X^{\vep_0}$ respectively. This fact can be verified by a formal computation of ${\vep_0}^{-1} (\dd X_{0t}^{\vep_0} - \dd \hat X_{0t} )$ and $\vep_0^{-1} ( \dd X_t^{\vep_0} - \dd \hat X_t)$ and then taking $\vep_0 \to 0$. To be precise, one can show that 
\begin{align*}
	\lim_{\vep_0 \to 0} \Ex \left [ \sup_{t \in [0, T]}  |\vep_0^{-1} (X_{0t}^{\vep_0} - \hat X_{0t}) - \hat V_{0t} |^2 \right ] &= 0, \quad \lim_{\vep_0 \to 0} \Ex \left [ \sup_{t \in [0, T]}  |\vep_0^{-1} (X_t^{\vep_0} - \hat X_t) - \hat V_t|^2 \right ] = 0.
\end{align*}
\item
Going back to \cref{eqn:gateauxJ0i}, we now rewrite the differences in $f_0$ and $g_0$ using the FTOC. Then, appealing to the regularity of coefficients provided in \Cref{ass:regularity}, we take $\vep_0$ to 0 as the limit suggests and note that $V_0$ and $V$ serve as derivatives of $X_0^{\vep_0}$ and $X^{\vep_0}$ respectively. Thus we obtain the form of the G\'ateaux derivative as
\begin{align}
\label{eqn:GateauxJ0iii}
\begin{split}
	&\lim_{\vep_0 \downarrow 0} \vep_0^{-1} \left [J_0(\alpha_0^{\vep_0}, \hat \alpha) - J_0(\hat \alpha_0, \hat \alpha) \right ] \\
	&= \Ex \Bigg [ \int_0^T \bigg (\partial_{x_0} f_0(t, \hat X_{0t}, \hat \alpha_{0t}, \LL(\hat X_t)) \cdot \hat V_{0t} + \partial_{\alpha_0} f_0(t, \hat X_{0t}, \hat \alpha_{0t}, \LL(\hat X_t)) \cdot \beta_{0t} \\ 
	&\quad + \tilde \Ex \left [ \partial_{\mu} f_0(t, \hat X_{0t}, \hat \alpha_{0t}, \LL(\hat X_t))(\tilde{\hat X}_t) \cdot \tilde{\hat{V}}_{0t} \right ] \bigg ) \dd t \\
	&\quad + \partial_{x_0} g_0(\hat X_{0T}, \LL(\hat X_T)) \cdot \hat V_{0T} +  \tilde \Ex[\partial_{\mu} g_0(\hat X_{0T}, \LL(\hat X_T))(\tilde{\hat{X}}_T) \cdot \tilde{\hat{V}}_{0T}] \Bigg].
\end{split}
\end{align}
\item
Our goal now is to eliminate the processes $\hat V_0$ and $\hat V$ in the preceding G\^ateaux derivative \cref{eqn:GateauxJ0iii} by means of an adjoint equation. We denote the adjoints for the major and minor player's state processes by $P_{0t}$ and $P_t$ respectively, and they are defined via the BSDEs:
\begin{align}
\label{eqn:necessarySMPBSDEmajor}
\begin{split}
	\dd P_{0t} &= - \partial_{x_0} H_0(t, \ubar{X}_t, \ubar{P}_t,  Q_t, \alpha_{0t}, \alpha_t, \LL(X_t)) \dd t \\
                &\quad + Q_{0t} \dd B_t, \\
	P_{0T} &=  \partial_{x_0} g_0(X_{0T}, \LL( X_T)) (X_T)
\end{split}
\end{align}
and
\begin{align}
\begin{split}
	\dd P_t &= - \Big [\partial_x H_0(t, \ubar{X}_t, \circ, P_t, Q_t, \alpha_{0t}, \alpha_t, \LL(X_t)) \\
			&\quad + \tilde \Ex[ \partial_{\mu} H_0(t, \ubar{\tilde X}_t, \ubar{\tilde P}_t, \tilde Q_t, \tilde \alpha_{0t}, \tilde \alpha_{t}, \LL(X_t))(X_t)] \Big] \dd t\\
			&\quad + Q_t \dd B_t, \\
	P_{T} &= \partial_{\mu} g_0(X_{0T}, \LL( X_T)) (X_T),
\end{split}
\end{align}
respectively. After noting that 
\begin{align*}
	\Ex [\hat V_{0T} \cdot \hat P_{0T} + \hat V_T \cdot \hat P_T ] = \Ex [\partial_{x_0} g_0(\hat X_{0T}, \LL(\hat X_T)) \cdot \hat V_{0T} +  \tilde \Ex[\partial_{\mu} g_0(\hat X_{0T}, \LL(\hat X_T))(\tilde{\hat X}_T) \cdot \tilde{\hat{V}}_{0T}]]
\end{align*}
and then expanding the preceding LHS through stochastic integration parts, it is now possible to rewrite the G\^ateaux derivative as 
\begin{align}
	&\lim_{\vep_0 \downarrow 0} \vep_0^{-1} \left [J_0(\alpha_0^{\vep_0}, \hat \alpha) - J_0(\hat \alpha_0, \hat \alpha) \right ] \nonumber \\
	&= \Ex\left [ \int_0^T \partial_{\alpha_0} H_0^R(t, \ubar{\hat{X}}_t, \ubar{\hat{P}}_t, \hat \alpha_{0t}, \hat \alpha_t, \LL(\hat X_t)) \cdot \beta_{0t} \dd t \right ] \geq 0 \label{eqn:necessarySMPmajorineq}
\end{align}
the last inequality being true due to the optimality of $(\hat \alpha_0, \hat \alpha)$. Due to $\beta_0 \in \AAbb_0$ being arbitrary, we obtain 
\begin{align}
\label{eqn:necessarySMPmajorineqpw}
    \Ex\left [ \partial_{\alpha_0} H_0^R(t, \ubar{\hat{X}}_t, \ubar{\hat{P}}_t, \hat \alpha_{0t}, \hat \alpha_t, \LL(\hat X_t)) \cdot \beta_{0t} \dd t \right ] \geq 0
\end{align}
a.e. in $t$. Finally, utilising the convexity of $\alpha_0 \mapsto H_0^R(t, \ubar{x}, \ubar{y}, \alpha_0, \alpha, \mu)$ and replacing $\beta_{0t}$ with $\alpha_0 - \hat \alpha_{0t}$ proves the claim.
\end{enumerate}

%-------------------------The minor player----------------------------------------
%---------------------------------------------------------------------------------
\subsection{The representative minor player's problem}

We now consider optimisation for the representation minor player. As this is the representative minor player's problem, we are considering state processes with the flow of probability measures $(\mu_t)_{\tinT}$ fixed, which is \crefrange{eqn:stateprocessmajor}{eqn:stateprocessminor}.

Consider a perturbation of their optimal strategy in the direction $\beta \in \AAbb$ given by $\alpha^{\vep} := \hat \alpha + \vep \beta$, which remains in $\AAbb$ for a sufficiently small $\vep$. Denote by $\hat X^\mu$ and $X^{\mu, \vep}$ for the minor player's state process \cref{eqn:stateprocessminor} along the strategies $(\hat \alpha_0, \hat \alpha)$ and $(\hat \alpha_0, \hat \alpha^\vep)$ respectively. Unlike the major player's problem, there is no need to consider the state process of the major player \cref{eqn:stateprocessmajor} along these two different strategies, as it does not depend on the minor player's strategy. Therefore, we we will denote by $\hat X_0^{\mu}$ for the major player's state process \cref{eqn:stateprocessmajor} along the strategy $\hat \alpha_0$.

Similar to the major player's problem, our goal is to study G\^ateaux derivative of $\alpha \mapsto J^{\mu}(\hat \alpha_0, \alpha)$ at $\hat \alpha$ in the direction $\beta \in \AAbb$, given by
\begin{align}
\label{eqn:gateauxJFi}
\begin{split}
	&\lim_{\vep \downarrow 0} \vep^{-1} \left [J^{\mu}(\hat \alpha_0, \alpha^{\vep}) - J^{\mu}(\hat \alpha_0, \hat \alpha) \right ]  \\
	&= \lim_{\vep \downarrow 0} \vep^{-1} \Ex \Bigg [ \int_0^T \left (f(t, \hat X_{0t}^{\mu}, X_t^{\mu, \vep}, \hat \alpha_{0t}, \alpha^{\vep}_t, \mu_t) - f(t, \hat X_{0t}^{\mu}, \hat X_t^{\mu}, \hat \alpha_{0t}, \hat \alpha_t, \mu_t) \right ) \dd t \\
	&\quad + g(\hat X^{\mu}_{0T}, X_T^{\mu, \vep},\mu_T) - g(\hat X_{0T}^{\mu}, \hat X_T^{\mu}, \mu_T) \Bigg ]. 
\end{split}
\end{align}
\begin{enumerate}[label = Step m(\roman*)]
\item
We now define the process $V^\mu$ taking values in $\reals^d$ through the SDE:
\begin{align*}
	\dd V_t^{\mu} &= \left ( \partial_x b(t, \hat X_{0t}^{\mu}, \hat X_t^{\mu}, \hat \alpha_{0t}, \hat \alpha_t, \mu_t) V_t^{\mu} + \partial_\alpha b(t, \hat X_{0t}^{\mu}, \hat X_t^{\mu}, \hat \alpha_{0t}, \hat \alpha_t, \mu_t) \beta_t \right )  \dd t  + \partial_x \sigma(t, X_t^{\mu}) V_t^{\mu} \dd B_t, \\ 
	V_0^{\mu} &= 0. 
\end{align*}
$V^\mu$ serves in a sense as a derivative of $\vep \mapsto X^{\mu, \vep}$. There is no need to consider such a process for $X_0^{\mu}$, since perturbations in $\alpha$ do not affect it. The fact that the above is the correct SDE for $V^\mu$ can be deduced through a formal computation of $\vep^{-1} ( \dd X_t^{\mu, \vep} - \dd \hat X_t)$ and then taking $\vep \to 0$. To be more precise, one can show that 
\begin{align*}
	\lim_{\vep \to 0} \Ex [ \sup_{t \in [0, T]} |\vep^{-1} (X_t^{\mu, \vep} - \hat X^\mu_t) - \hat V_t^\mu|^2] = 0.
\end{align*}
\item
Going back to \cref{eqn:gateauxJFi}, we rewrite the differences in $f$ and $g$ using the FTOC. Then, appealing to the regularity of the coefficients guaranteed by \Cref{ass:regularity}, we take the limit as $\vep \to 0$. Noting that $V^{\mu}$ serves as a derivative of $X^{\mu, \vep}$, we obtain the form of the G\'ateaux derivative as
\begin{align*}
	&\lim_{\vep \downarrow 0} \vep^{-1} \left [J(\hat \alpha_0, \alpha^{\vep}) - J(\hat \alpha_0, \hat \alpha) \right ]
	\\ &\Ex \Bigg [ \int_0^T \left ( \partial_x f(t, \hat X_{0t}^{\mu}, \hat X_t^{\mu}, \hat \alpha_{0t}, \hat \alpha_t, \mu_t) \cdot \hat V_t^{\mu} + \partial_{\alpha} f(t, \hat X_{0t}^{\mu}, \hat X_t^{\mu}, \hat \alpha_{0t}, \hat \alpha_t, \mu_t) \cdot \beta_t \right ) \dd t  \\ 
	&\quad + \partial_x g(\hat X_{0T}^{\mu}, \hat X_T^{\mu}, \mu_T) \cdot \hat V^{\mu}_{T} \Bigg ].
\end{align*}	
\item
Our goal now is to eliminate the variable $\hat V^\mu$ in the preceding G\^ateaux derivative by means of an adjoint equation. We define the adjoint process $(Y^{\mu}_{t}, Z^{\mu}_t)$ via the BSDE:
\begin{align*}
	\dd Y^{\mu}_t &= - \partial_x H(t, X^{\mu}_{0t}, X^{\mu}_t,Y_{t}^{\mu}, Z_{t}^{\mu},  \alpha_{0t}, \alpha_t, \mu_t) \dd t\\
			&\quad + Z^{\mu}_t \dd B_t, \\
		Y^{\mu}_T &=  \partial_x g(X_{0T}^{\mu}, X^{\mu}_T, \mu_T).
\end{align*}
After noting that 
\begin{align*}
	\Ex [\hat V_T^\mu \cdot Y^\mu_T ] = \Ex [ \partial_x g(\hat X_{0T}^{\mu}, \hat X_T^{\mu}, \mu_T) \cdot \hat V^{\mu}_{T} ]
\end{align*}
and then expanding the preceding LHS through stochastic integration parts, it is now possible to rewrite the G\^ateaux derivative as 
\begin{align}
	&\lim_{\vep \downarrow 0} \vep^{-1} \left [J^{\mu}(\hat \alpha_0, \alpha^{\vep}) - J^{\mu}(\hat \alpha_0, \hat \alpha) \right ] \nonumber \\
	&= \Ex \left [\int_0^T \partial_{\alpha} H^R(t, \hat{\ubar{X}}_t^{\mu}, \hat Y^\mu_t, \hat Z^\mu_t, \hat \alpha_{0t}, \hat \alpha_t, \mu_t) \cdot \beta_t \dd t \right ] \geq 0 \label{eqn:necessarySMPminorineq}
\end{align}	
the last inequality being true due to the optimality of $(\hat \alpha_0, \hat \alpha)$. Since $\beta \in \AAbb$ is arbitrary, we obtain 
\begin{align}
\label{eqn:necessarySMPminorineqpw}
	\partial_\alpha H^R(t, \hat X_{0t}^\mu, \hat{\ubar{X}}_t^\mu, \hat Y^\mu_t, \hat Z^\mu_t, \hat \alpha_{0t}, \hat \alpha_t, \mu_t) 
\end{align}
$\dd t \times \dd \Pro$ a.e. Lastly, utilising the convexity of $\alpha \mapsto H(t, \ubar{x}, y, \alpha_0, \alpha, \mu)$ and replacing $\beta$ with $\alpha - \hat \alpha_t$ proves the claim.
\end{enumerate}

\begin{remark}
A key difference in the major player's problem versus the minor player's problem is that we can only remove the integral w.r.t. $\dd t$, and cannot remove the expectation, in the inequality \cref{eqn:necessarySMPmajorineq} for the major player. For the minor player, we can remove both the integral w.r.t. $\dd t$ and expectation in \cref{eqn:necessarySMPminorineq}. As a consequence, the inequality \cref{eqn:necessarySMPmajorineqpw} holds $\dd t$ a.e., whereas the inequality \cref{eqn:necessarySMPminorineqpw} holds $\dd t \times \dd \Pro$ a.e. The reason for this is due to the fact that elements of $\AAbb_0$ are deterministic processes, whereas elements of $\AAbb$ are processes progressively measurable w.r.t. $(\FF_t)_{\tinT}$. This observation is crucial for the statements of \Cref{ass:minimisersofhamiltonians} and \Cref{thm:sufficientSMPMm}.
\end{remark}

\end{document}